\documentclass[11pt]{amsart}
\usepackage{amsmath}
\usepackage{amssymb}
\usepackage{amscd}
\usepackage[all]{xy}

\usepackage{lipsum,graphicx}
\usepackage{hyperref}

\usepackage{times}
\usepackage{xmpmulti}
\usepackage{graphicx}

\newtheorem{Theorem}{Theorem}[section]
\newtheorem{Lemma}[Theorem]{Lemma}

\newtheorem{Proposition}[Theorem]{Proposition}
\newtheorem{Remark}[Theorem]{Remark}

\newtheorem{Definition}[Theorem]{Definition}

\let\epsilon\varepsilon
\let\kappa=\varkappa

\def \a {\alpha}

\def \y {\mathrm{y}}
\def \x {\mathrm{x}}

\def \l {\ell}

\begin{document}

\title[Toroidalization of Locally Toroidal Morphisms]
{Toroidalization of Locally Toroidal Morphisms}

\thanks{This research was supported by a grant from IPM}
  
\address{Razieh Ahmadian, School of Mathematics, Institute for Research in Fundamental Sciences (IPM), P.O.Box 19395-5746, Tehran, Iran.}

\email{ahmadian@ipm.ir}

\author[R.\ Ahmadian]{R.\ Ahmadian}

\maketitle

\begin{abstract}
	\noindent 
	The problem of toroidalization is to construct a toroidal lifting of a dominant morphism $\varphi:X\to Y$ of algebraic varieties by blowing up in the target and domain. This paper contains a solution to this problem when $\varphi$ is locally toroidal.
\end{abstract}

\tableofcontents

%%%%%%%%%%%%%%%%%****SECTION : INTRODUCTION****%%%%%%%%%%%%%%%%%%%%%%%%%%
\section{Introduction}

Suppose that $\varphi:X\to Y$ is a dominant morphism of algebraic varieties. An interesting question is whether we can construct a commutative diagram 
\begin{equation}\label{generalCD}
\begin{CD}
\widetilde X @>\tilde{\varphi}>> \widetilde Y\\
@ V\lambda VV @VV \pi V\\
X @>\varphi>> Y
\end{CD}
\end{equation}
where $\lambda,\pi$ are proper birational morphisms, and the new morphism $\tilde{\varphi}:\widetilde X\to \widetilde Y$ is rather well-understood, e.g., $\tilde\varphi$ is a locally monomial morphism. As (\ref{generalCD}) gives the factorization $\pi\circ\tilde\varphi\circ\lambda^{-1}$ of $\varphi$, the answer of such a question appears in many applications, most notably in simplifying the problem of strong factorization of morphisms \cite[Question F']{H}.\\

In particular, if we want $\tilde\varphi$ to be toroidal, and $\lambda,\pi$ to be sequences of blow ups of nonsingular centers, we obtain the problem of toroidalization proposed in \cite[problem 6.2.1]{AKMW}. This problem has been addressed in many research articles such as \cite{C,CK,C3,Ha,ADK,C5,A2,...}.  We note that toroidalization does not exist in positive characteristic $p>0$, even for maps of curves, for instance, $y = x^p + x^{p+1}$ \cite{C3}. In addition, this problem can be reduced to the case of morphisms of nonsingular varieties by resolution of singularities in characteristic zero -- see \cite{H}, or any of the simplified proofs including \cite{BM}, \cite{BEV}, and \cite{EH}. When $Y$ is a curve, toroidalization is an immediate consequence of the embedded resolution of hypersurface singularities. We also know a dominant morphism $\varphi:X\to Y$ with $\dim Y\leqslant\dim X\leqslant 3$ can be made toroidal (see \cite{C3} for the full story).\\
 
There are also nice results related to toroidalization in arbitrary dimensions of $X$ and $Y$. The existence of (\ref{generalCD}) has been proved in \cite[Theorem 1.2]{AK} where $\lambda,\pi$ are modifications (arbitrary birational morphism) and $\tilde\varphi$ has a toroidal structure. In \cite[Theorem]{C1}, it is shown that a diagram (\ref{generalCD}) can be constructed where $\lambda,\pi$ are locally products of blowing ups of nonsingular centers and $\tilde\varphi$ is locally toroidal (We note that the morphisms $\lambda,\pi$ and $\tilde\varphi$ may not be separated). The notion of locally toroidal morphisms has been made precise in \cite{Ha}, where a toroidalization of such a morphism from an $n$--fold to a surface has been constructed. In this paper, which is a sequel to \cite{Ha,A1,A2,A3}, we prove the existence of toroidalization for locally toroidal morphisms in arbitrary dimensions of $X$ and $Y$. This case of the toroidalization problem is among patching type problems which appeared in \cite{Z} and \cite{T}.\\

Throughout this paper, $\mathcal K$ is an algebraically closed field of characteristic zero, and a variety is a quasi--projective variety over $\mathcal K$. A normal variety $X$ is \textit{toroidal} if it contains a nonsingular Zariski open subset $U\subset X$ with the property that for each $p\in X$, there exists a neighborhood $U_p$ of $p$, and an affine toric variety $X_{\sigma}$ with an \'{e}tale morphism $\pi:U_p\to X_{\sigma}$ such that $\pi^{-1}(T)=U\cap U_p$ where $T$ is the algebraic torus in $X_{\sigma}$. When $X$ is nonsingular\footnote{Due to the existence of resolution of singularities in characteristic zero \cite{H}, the problem of toroidalization can be reduced to the case of morphisms of nonsingular varieties.}, any simple normal crossings (SNC) divisor $D\subset X$, letting $U=X\setminus D$, specifies a toroidal structure on $X$ -- see Proposition \ref{SNC div Specifies a Tor structure}. A dominant morphism $\varphi:X\to Y$ of nonsingular varieties is \textit{toroidal} if there exist SNC divisors $D\subset X$ and $E\subset Y$ such that $D=\varphi^{-1}(E)$, and $\varphi$ is locally given by monomials in appropriate \'{e}tale local parameters on X. We will supply more technical details and references in the next section, and we just provide the defined of locally toroidal morphisms.

\begin{Definition}\label{Def: Locally Toroidal Morphism}
 	Let $\varphi:X\to Y$ be a dominant morphism of nonsingular varieties. Suppose that there exist finite open covers $\{U_{\a}\}_{\a\in I}$ of $X$ and $\{V_{\a}\}_{\a\in I}$ of $Y$, and SNC divisors $D_{\a}\subset U_{\a}$ and $E_{\a}\subset V_{\a}$ for each $\a\in I$, such that 
 	\begin{enumerate}
 		\item $\varphi_{\a}:=\varphi|_{U_{\a}}:U_{\a}\to V_{\a}$,
 		\item $D_{\a}=\varphi_{\a}^{-1}(E_{\a})$, and
 		\item $\varphi_{\a}:U_{\a}\setminus D_{\a} \to V_{\a}\setminus E_{\a}$ is smooth,
 	\end{enumerate}
 for all $\a\in I$. We say that $\varphi:(X,U_{\a},D_{\a})_{\a\in I}\to (Y,V_{\a},E_{\a})_{\a\in I}$ is \textbf{\textit{locally toroidal}} if for each $\a\in I$, $\varphi_{\a}:U_{\a}\to V_{\a}$ is toroidal with respect to $D_{\a}$ and $E_{\a}$ (c.f. \cite [Definition 1.3]{Ha}).
 \end{Definition}

We will give a characterization of toroidal morphisms in Theorem \ref{TF}, which is a generalization of \cite[Lemma 19.3]{C}, \cite[Lemma 4.2]{CK} and \cite[Lemma 3.4]{A1} to arbitrary dimensions of $X,Y$. Besides using this result, applying some classic resolution theorems, namely embedded resolution of singularities and principalization of ideal sheaves, enable us to develop the strategy used in \cite{Ha,A1,A2} and to construct toroidalization of locally morphisms in arbitrary dimensions. 

\begin{Theorem}\label{Main Theorem}
	Suppose that $\varphi:(X,U_{\a},D_{\a})_{\a\in I}\to (Y,V_{\a},E_{\a})_{\a\in I}$ is a locally toroidal morphism of nonsingular varieties. There exists a commutative diagram
	\begin{equation}\label{Toroidalization of LocTor}
		\begin{CD}
	\widetilde{X}@>\widetilde\varphi>>\widetilde{Y}\\
	@V\lambda VV @VV\pi V\\ X@>>\varphi>Y
	\end{CD}
   \end{equation}
	such that $\lambda:\widetilde{X}\rightarrow X$ and $\pi:\widetilde{Y}\rightarrow Y$ are sequences of blowups with nonsingular centers, $\widetilde{X}$ and $\widetilde{Y}$ are nonsingular, and there exists SNC divisor $\widetilde E$ on $\widetilde{Y}$ such that $\widetilde D:=\widetilde\varphi^{-1}(\widetilde E)$ is a SNC divisor on $\widetilde X$, and $\widetilde\varphi$ is toroidal with respect to $\widetilde E$ and $\widetilde D$.
\end{Theorem}
Most of the techniques needed to construct the diagram (\ref{Toroidalization of LocTor}) are provided in section \ref{Resolution Techniques}, and we will prove this theorem, which is a generalization of \cite[Theorem 4.2]{Ha}, \cite[Theorem 3.19]{A2} and \cite[Theorem 6.3]{A1} to arbitrary dimensions of $X$ and $Y$, in the last section. 

%%%%%%%%%%%%****SECTION : TROIDAL VARIETIES & TOROIDAL MORPHISMS****%%%%%%%%%%%%%%%%%%%

\section{Focus on Toroidal Varieties and Toroidal Morphisms}

The idea of \textit{toroidal varieties}, called \textit{toroidal embeddings} originally, was introduced and developed by Kempf, Knudsen, Mumford and Saint-Donat in the Springer lecture notes \textit{Toroidal Embeddings I}, published in 1973 \cite{KKMS}. In this section, we study the structure of nonsingular toroidal varieties and morphisms. We will observe how a simple normal crossings (SNC) divisor specifies a toroidal structure on a nonsingular variety in Proposition \ref{SNC div Specifies a Tor structure}. We will also give a characterization of toroidal morphisms in Theorem \ref{TF}.

\begin{Definition}[Definition 1, page 54 \cite{KKMS}]
	\label{Def:ToroidalVariety}
	Suppose that $X$ is a normal $d$-dimensional $\mathcal K$--variety, and let $U\subset X$ be a nonsingular Zariski open subset. We say $U$ specifies a toroidal structure on $X$, and $X$ is called a \textbf{toroidal variety} if for every closed point $p\in X$ there exist an affine toric variety $X_{\sigma}$ containing a $d$-dimensional torus $T$, a point $p'\in X_{\sigma}$, and an isomorphism of $\mathcal K$--local algebras
	$$
	\iota_p:\hat{\mathcal O}_{X,p}\to\hat{\mathcal O}_{X_{\sigma},p'}
	$$
	such that $\mathcal I_{X\setminus U}\hat{\mathcal O}_{X,p}$ corresponds to $\mathcal I_{X_{\sigma}\setminus T}\hat{\mathcal O}_{X_{\sigma},p'}$  under this isomorphism. The pair $(X_{\sigma},p')$ is called a \textbf{local model} at $p\in X$.
	%(the ideal in $\hat{\mathcal O}_{X,p}$ generated by the ideal of $X\setminus U$)
	%(the ideal in $\hat{\mathcal O}_{X_{\sigma},p'}$ generated by the ideal of $X_{\sigma}\setminus T$)
\end{Definition}
From the context of this definition and the fact that $X_{\sigma}\setminus T$ is purely one codimensional, it can be seen that $X\setminus U$ is a purely one codimensional subscheme of $X$ as well. Writing $X\setminus U=\cup_i D_i$ where $D_i$ are irreducible subvarieties of codimension one in $X$, the toroidal structure on $X$ can be equivalently specified by the reduced Weil divisor $\sum_i D_i$ supported in $X\setminus U$. In particular, any simple normal crossings divisor $D_X$ on $X$, letting $U=X\setminus\mathrm{Supp}\,D_X$, specifies a toroidal structure on $X$, usually denoted by $(X,D_X)$ -- see \cite [Definition 4.3, page 21]{C3} where we can also find the definition of toroidal morphisms.\\

Recall that an effective divisor $D=\sum d_iD_i$, where $d_i\in\mathbb N$ and $D_i$ are irreducible codimension one subvarieties of $X$, is \textbf{\textit{simple normal crossings} (SNC)} if at each $p\in X$ there exist regular parameters $x_1,\dots,x_d$ in $\mathcal O_{X,p}$, and natural numbers $a_1,\dots,a_d$ such that 
$$
\mathcal I_{D,p}=x_1^{a_1}\cdots x_d^{a_d}\mathcal O_{X,p}
$$
where $\mathcal I_D\subset\mathcal O_X$ is the ideal sheaf of $\mathrm{Supp}\,D=\bigcup D_i$. Since the multiplicities $d_i$ of the irreducible components $D_i$ do not affect the definition of being SNC, we may replace $D$ with the reduced divisor $D_{\mathrm{red}}=\sum D_i$. If $\varphi:X\to Y$ is a morphism of varieties, and $E$ is a (reduced) Cartier divisor on $Y$, then $\varphi^{-1}(E)$ will denote the reduced divisor $\varphi^*(E)_{\rm red}$ on $X$.

\begin{Definition}\cite [Definition 4.3, page 21]{C3}
	Suppose that $\varphi:(X,D_X)\to(Y,D_Y)$ is a dominant morphism of toroidal varieties with SNC divisors $D_Y$ and $D_X=\varphi^{-1}(D_Y)$. The morphism $\varphi$ is called \textbf{toroidal} (with respect to $D_Y$ and $D_X$) if at all $p\in X$ with $q=\varphi(p)$, there exist local models $(X_{\sigma},p')$ for $X$ at $p$ and $(Y_{\tau},q')$ for $Y$ at $q$, and a toric morphism $\phi:X_{\sigma}\to Y_{\tau}$ such that the following diagram commutes:
	\begin{equation}\label{CD of Local Rings}
	\begin{CD}
	{\hat{\mathcal{O}}_{X,p}}@>\simeq>>\hat{\mathcal{O}}_{X_{\sigma},p'}\\
	@A\hat{\varphi}^{*}AA @AA\hat{\phi}^{*}A\\
	{\hat{\mathcal{O}}_{Y,q}}@>\simeq>>\hat{\mathcal{O}}_{Y_{\tau},q'}.
	\end{CD}
	\end{equation}
\end{Definition}

The first step in studying the toroidal structure determined by a SNC divisor $D_X$ on a nonsingular variety $X$ is to recognize local models at closed points of $X$. For this purpose, we recall some basic tools in toric geometry. We first give a brief introduction to the construction of nonsingular affine toric varieties as we will study toric morphisms between them (equivariant maps that respect the torus actions). Our primary reference is \cite{CLS}.\\ 

 A \textit{\textbf{toric variety}} is a normal variety $\mathfrak X$ that contains an \textit{algebraic torus} $T$ as a Zariski open subset such that the natural action of $T$ on itself extends to an \textit{algebraic action} of $T$ on $\mathfrak X$, i,e., an action $T\times \mathfrak X\to \mathfrak X$ which is given by a morphism of algebraic varieties. In the $d$-dimensional affine space $\mathbb A^d_{\mathcal K}=\mathrm{Spec}\,\mathcal K[u_1,\dots,u_d]$ over $\mathcal K$, the affine variety $G=\mathrm{Spec}\,\mathcal K[u_1,\dots,u_d]_{u_1\cdots u_d}$, i.e., the Zariski open subset given by $u_1\cdots u_d\neq 0$, is an \textit{algebraic group}\footnote{A variety endowed with the structure of a group is called an \textbf{algebraic group} if the maps $\mu:G\times G\to G$, $\mu(x,y)=xy$, and $\iota:G\to G$, $\iota(x)=x^{-1}$, are morphisms of algebraic varieties \cite{Hu}.} under componentwise multiplication. An \textit{\textbf{algebraic torus}} $T$ is an affine variety isomorphic to $G$, where $T$ inherits a group structure from the isomorphism. (We note that the set of $\mathcal K$-\textit{valued points}\footnote{A $\mathcal K$-\textbf{valued point} $p$ in a $\mathcal K$-scheme $U$ is given by a morphism $\mathrm{Spec}\,\mathcal K\to U$.} of $T$ is actually isomorphic to $(\mathcal K^{\times})^d$ where $\mathcal K^{\times}=\mathcal K\setminus\{0\}$ is the multiplicative group of $\mathcal K$). The \textit{characters}\footnote{A \textbf{character} of a group $H$ is a group homomorphism from $H$ to the multiplicative group of a field \cite{}.} $\chi:T\to\mathcal K^{\times}$ of $T$ form a free $\mathbb Z$-module $M$ of rank equal to the dimension of $T$, called the \textit{\textbf{character lattice}} of $T$.\\

Denoting by $N$ the dual lattice $\mathrm{Hom}_{\mathbb Z}\,(M,\mathbb Z)$ of the character lattice $M$ of $T$, when $\mathfrak X$ is affine, there exists a \textit{strongly convex rational polyhedral cone} $\sigma$ in the real vector space $N_{\mathbb R}=N\otimes_{\mathbb Z}\mathbb R$ such that $\mathfrak X=X_{\sigma}=\mathrm{Spec}\,\mathcal K[S_{\sigma}]$ where $\mathcal K[S_{\sigma}]$ is the \textit{semigroup algebra}\footnote{Given a semigroup $(S,\ast)$, the \textbf{semigroup algebra} $\mathcal K[S]$ is the $\mathcal K$-vector space with $S$ as a basis and multiplication determined by $\ast$ in $S$. Precisely, $\mathcal K[S]=\{\sum_{m\in S} \alpha_m m \ |\ \alpha_m\in\mathcal K \mbox{ and }\alpha_m=0 \mbox{ for all but finitely many }m\in S \}$, and $(\sum_{m\in S} \alpha_m m)(\sum_{n\in S} \beta_n n)=\sum_{m,n\in S} \alpha_m\beta_n(m\ast n)$ \cite{Ok}.} of $S_{\sigma}=\sigma^{\vee}\cap M$, and $\sigma^{\vee}$ is the \textit{dual cone} in $M_{\mathbb R}=M\otimes_{\mathbb Z}\mathbb R$ \cite[Theorem 1.3.5]{CLS}. In particular, an affine toric variety $X_{\sigma}$ is nonsingular if and only if the cone $\sigma$ is \textit{regular} \cite[Theorem 1.3.12]{CLS}. We will make this precise bellow.\\

A polyhedral cone $\sigma\subset N_{\mathbb R}$ is called \textit{\textbf{rational}} if $\sigma=\mathrm{Cone}\,(S)=\{\ \sum_{v\in S}a_v v\ |\ a_v\in\mathbb R^{\geqslant 0}\}$ for some finite set $S\subset N$ of vectors in the lattice, and it is \textit{\textbf{regular}} when $S$ can be extended to a $\mathbb Z$-basis of $N$. \textit{\textbf{Strong convexity}} means that $\sigma$ contains no positive-dimensional subspace of $N_{\mathbb R}$, equivalently, the origin $\{0\}$ is a face of it. The \textit{\textbf{dimension}} of $\sigma$ is the dimension of the smallest subspace of $N_{\mathbb R}$ containing $\sigma$, that is the linear subspace $\mathbb R\sigma=\sigma+(-\sigma)$ spanned by $\sigma$. We have a canonical $\mathbb Z$-bilinear paring $\langle,\rangle:M\times N\to\mathbb Z$ which gives a canonical $\mathbb R$-bilinear paring $\langle,\rangle:M_{\mathbb R}\times N_{\mathbb R}\to\mathbb R$ by scalar extension to $\mathbb R$. The \textit{\textbf{dual cone}} of $\sigma$ is
$$
\sigma^{\vee}=\{m\in M_{\mathbb R}| \langle m,v\rangle\geqslant 0 \mbox{ for all } v\in\sigma\}
$$
which is a convex rational polyhedral cone in $M_{\mathbb R}$.\\

As already noted, according to \cite[Theorem 1.3.12]{CLS}, a nonsingular $d$-dimensional affine toric variety $X_{\sigma}=\mathrm{Spec}\,\mathcal K[S_{\sigma}]$ over $\mathcal K$ is constructed from a lattice $N\cong\mathbb Z^d$, and a regular strongly convex rational polyhedral cone $\sigma\subseteq N_{\mathbb R}$. Let $\{e_1,\dots,e_d\}$ and $\{e_1^*,\dots,e_d^*\}$ be $\mathbb Z$-bases of $N$ and $M$ respectively. Since $\sigma$ is regular, writing $\dim\sigma=n\leqslant d$, we may assume
\begin{align*}
\sigma        &=\mathrm{Cone}\,(e_1,\dots,e_n),\mbox{ and hence }\\
\sigma^{\vee} &=\mathrm{Cone}\,(e_1^*,\dots,e_n^*,\pm e_{n+1}^*,\dots,\pm e_d^*).
\end{align*}

Then the additive semigroup $S_{\sigma}=\sigma^{\vee}\cap M$ is generated by\footnote{More generally, any semigroup $S_{\sigma}=\sigma^{\vee}\cap M$ obtained from a rational polyhedral cone $\sigma$ is finitely generated \cite[Proposition 1.2.17 (Gordan's Lemma)]{CLS}.}
$$
\mathfrak{G_{\sigma}}=\{e^*_1,\dots,e^*_n,\pm e^*_{n+1},\dots,\pm e^*_d\}
$$
which is the sub-semigroup of the character lattice $M$, also denoted by $\mathbb N\mathfrak{G_{\sigma}}$. We note that, as a semigroup, $M$ is generated by $\{\pm e^*_1,\dots,\pm e^*_d\}$, and the $\mathcal K$-algebra homomorphism induced by $e_i^*\mapsto u_i$ defines an isomorphism
$$
\mathcal K[M]\cong\mathcal K[u_1^{\pm 1},\dots,u_d^{\pm 1}]=\mathcal K[u_1,\dots,u_d]_{u_1\cdots u_d}
$$
of the semigroup algebra $\mathcal K[M]$ and the ring of Laurent polynomials in $d$ variables. In particular, a lattice point $m=\sum_{i=1}^d a_ie_i^*\in M$, with $a_1,\dots,a_d\in\mathbb Z$, corresponds to the Laurent monomial $u_1^{a_1}\cdots u_d^{a_d}$. In addition, under this isomorphism, the semigroup algebra $\mathcal K[S_{\sigma}]$, which is a subalgebra of $\mathcal K[M]$, corresponds to
\begin{equation}\label{SGAlg}
\mathcal K[S_{\sigma}]\cong\mathcal K[u_1,\dots,u_n,u_{n+1}^{\pm 1},\dots,u_d^{\pm 1}]=\mathcal K[u_1,\dots,u_d]_{u_{n+1}\cdots u_d}.
\end{equation}

In this way the nonsingular affine toric variety $X_{\sigma}=\mathrm{Spec}\,\mathcal K[S_{\sigma}]$, containing the algebraic torus $T=\mathrm{Spec}\,\mathcal K[M]$, is thoroughly recognizable by $d=\dim X_{\sigma}$ and $n=\dim\sigma$.

\subsection{Local Models}
Now suppose $X$ is a nonsingular $\mathcal K$-variety of dimension $d$, and $D_X$ is a reduced SNC divisor on $X$. We notice that $D_X$ classifies closed points of $X$ as follows:
\begin{Definition}\label{per parameters}
	A closed point $p\in X$ is called an \textbf{n-point} for $D_X$ if $p$ lies in exactly $n$ irreducible components of $D_X$. We have $0\leqslant n\leqslant d=\dim X$.
	
	We say that $x_1,\dots,x_d$ are \textbf{(formal) permissible parameters} at $p$ (for $D_X$) if $x_1,\dots,x_d$ are regular parameters in $\hat{\mathcal O}_{X,p}$ and $x_1\cdots x_n=0$ is a (formal) local equation of $D_X$ if $p$ is an $n$-point. We say that permissible parameters $x_1,\dots,x_d$ are \textbf{algebraic} if $x_1,\dots,x_d\in\mathcal O_{X,p}$.
\end{Definition}

Suppose that $p$ is an $n$-point for $D_X$, for some $0\leqslant n\leqslant d$, and $x_1,\dots,x_d\in\hat{\mathcal O}_{X,p}$ are permissible parameters at $p$, i.e., $p$ corresponds to the maximal ideal
$\mathfrak m_p=\langle x_1,\dots,x_d\rangle\subset\hat{\mathcal O}_{X,p}$, and $\mathcal I_{D_X}\hat{\mathcal O}_{X,p}=\langle x_1\cdots x_n\rangle$
where $\mathcal I_{D_X}\subset\mathcal O_X$ is the ideal sheaf of $\mathrm{Supp}\,D_X$.
%As a $\mathcal K$-valued point of the affine formal neighborhood $\mathrm{Spec}\,\hat{\mathcal O}_{X,p}$ in $X$, $p$ can be represented\footnote{A $\mathcal K$-valued point $p$ in a $\mathcal K$-scheme $U$ is given by a morphism $\mathrm{Spec}\,\mathcal K\to U$.} by a $\mathcal K$-algebra homomorphism $\gamma_p:\hat{\mathcal O}_{X,p}\to\mathcal K$. Writing $\alpha_i:=\gamma_p(x_i)$, we have $\alpha_1=\cdots=\alpha_n=0$ and $\alpha_{n+1},\dots,\alpha_d\in\mathcal K^{\times}$.
%%%%%%%%%%%%%%%%%%%%%%%%%%%%%%%%%%%%%%%%%%%%%%%%%%%%%%%%%%%%%%%%%%%%%%%%%%%%%%%%%%%%%%%%%%%%%%%%%%%%%%%%%%%%%%%%%%%%%%%%%%%%%%%%%%%%%%%%%%%%%%%%%%%%%%%%%%%%%%%%%%%%%%%%%%%%%%%%%%%%%%%%%%%%%%
\begin{Proposition}\label{SNC div Specifies a Tor structure}
	The nonsingular affine toric variety
	$$
	X_{\sigma}=\mathrm{Spec\,}\mathcal K[S_{\sigma}]=\mathrm{Spec\,}\mathcal K[u_1,\dots,u_n,u_{n+1}^{\pm 1},\dots,u_d^{\pm 1}]
	$$
	with the closed point $p'$ corresponding to the maximal ideal
	$$
	\mathfrak m_{p'}=\langle u_1,\dots,u_n,u_{n+1}-\alpha_{n+1},\dots,u_d-\alpha_d\rangle\subset\mathcal K[S_{\sigma}],\ \ \ \alpha_{n+1},\dots,\alpha_d\in\mathcal K\setminus\{0\},
	$$
	is a local model of $X$ at $p$. Furthermore, $X_{\sigma}$ is unique up to isomorphism.
\end{Proposition}

\begin{proof}
	First we observe that $u_1,\dots,u_d$ are \textit{uniformizing parameters}\footnote{Suppose that $U$ is an open subset of a nonsingular variety $X$. Elements $f_1,\dots,f_n\in\Gamma(U,\mathcal O_X)$ are called \textbf{uniformizing parameters} on $U$ if $df_1,\dots,df_n$ are a free basis of $\Omega_{X}|_U$ \cite[Definition 14.17, page 171]{C4}.} on $X_{\sigma}$ since, denoting by $A$ the polynomial ring $\mathcal K[u_1,\dots,u_d]$, and writing $S=\{(u_{n+1}\cdots u_d)^k\ |\ k\geqslant 0\ \}$, we have $\mathcal K[S_{\sigma}]=S^{-1}A$ and the module of K\"{a}hler derivations $\Omega_{S^{-1}A/\mathcal K}$ is the free $S^{-1}A$-module of rank $d$ generated by $du_1,\dots,du_d$. This follows from the natural isomorphism of $S^{-1}A$-modules $\Omega_{S^{-1}A/\mathcal K}\cong S^{-1}\Omega_{A/\mathcal K}$ \cite[Proposition 16.9, page 394]{Ei}, and the fact that $du_1,\dots,du_d$ are a free basis of $\Omega_{A/\mathcal K}$ \cite[Proposition 16.1, page 385]{Ei}.
	
	We also notice that, given $\alpha_{n+1},\dots,\alpha_d\in\mathcal K^{\times}$, assigning $u_i\mapsto\left\{
	\begin{array}{ll}
	0, & 1\leqslant i\leqslant n \\
	\alpha_i, & n+1\leqslant i\leqslant d
	\end{array}
	\right.$ induces a $\mathcal K$-algebra homomorphism
	$$
	\Gamma:\mathcal O_{X_{\sigma}}=\mathcal K[u_1,\dots,u_n,u_{n+1}^{\pm 1},\dots,u_d^{\pm 1}]\to\mathcal K.
	$$
	Clearly, the $\mathcal K$-valued point $p'$ given by $\Gamma$ corresponds to the maximal ideal
	$$
	\mathfrak m_{p'}=\langle u_1,\dots,u_n,u_{n+1}-\alpha_{n+1},\dots,u_d-\alpha_d\rangle\subset\mathcal K[S_{\sigma}].
	$$
	In addition, given a regular function $f:X_{\sigma}\to\mathcal K$ in $\mathcal O_{X_{\sigma}}$, we have $f(p')=\Gamma(f)$. In particular, we have $u_1(p')=\cdots=u_n(p')=0$, and $u_{i}(p')=\alpha_i$, for $n+1\leqslant i\leqslant d$.
	
	Thus, by \cite[Proposition 14.18, page 171]{C4}, $u_1,\dots,u_n,u_{n+1}-\alpha_{n+1},\dots,u_d-\alpha_d$ are regular parameters in $\mathcal O_{X_{\sigma},p'}$ since $u_1,\dots,u_d$ are uniformizing parameters on $X_{\sigma}$. Furthermore,
	$$
	\hat{\mathcal O}_{X_{\sigma},p'}=\mathcal K[[u_1,\dots,u_n,u_{n+1}-\alpha_{n+1},\dots,u_d-\alpha_d]],
	$$
	by \cite[Proposition 15.13, page 178]{C4}. In addition, the ideal in $\hat{\mathcal O}_{X_{\sigma},p'}$ generated by the ideal of $X_{\sigma}\setminus T=X_{\sigma}\cap \mathbb V(\langle u_1\cdots u_d\rangle)$ is the monomial ideal $\langle u_1\cdots u_n\rangle$, i.e., $\mathcal I_{X_{\sigma}\setminus T}\hat{\mathcal O}_{X_{\sigma},p'}=\langle u_1\cdots u_n\rangle$.
	
	By the following Remark\ref{Remark1}, there exists a unique $\mathcal K$-local algebra homomorphism
	$$\iota_p:\hat{\mathcal O}_{X,p}\to\hat{\mathcal O}_{X_{\sigma},p'}\ \mbox{ such that }\ \iota_p(x_i)=\left\{
	\begin{array}{ll}
	u_i & \ \ \ \ \ \ \ 1\leqslant i\leqslant n, \\
	u_i-\alpha_i & n+1\leqslant i\leqslant d.
	\end{array}
	\right.
	$$
	We see at once that, under this isomorphism, $\mathcal I_{D_X}\hat{\mathcal O}_{X,p}$ corresponds to $\mathcal I_{X_{\sigma}\setminus T}\hat{\mathcal O}_{X_{\sigma},p'}$. Therefore, $(X_{\sigma},p')$ is a local model for $X$ at $p$, and uniqueness follows from \cite[Theorem 1.3.12]{CLS}.
\end{proof}

\begin{Remark}\cite[Exersice 3., page 180]{C4}\label{Remark1}
	A $\mathcal K$-local algebra homomorphism between two power series rings $\iota:\mathcal K[[x_1,\dots,x_m]]\to\mathcal K[[y_1,\dots,y_n]]$ is uniquely determined by a collection of elements $h_1,\dots,h_m$ in the maximal ideal of $\mathcal K[[y_1,\dots,y_n]]$ such that for $1\leqslant i \leqslant m$, $\iota(x_i)=h_i$.
\end{Remark}

Proposition \ref{SNC div Specifies a Tor structure} clarifies how a SNC divisor $D_X$ equips a nonsingular variety $X$ with a toroidal structure. If $\dim X=d$ and $p\in X$ is an $n$-point for $D_X$, $0\leqslant n\leqslant d$, then a local model at $p$ is the $d$-dimensional nonsingular affine toric variety $X_{\sigma}$ coming from a lattice $N\cong\mathbb Z^d$, and a regular strongly convex rational polyhedral cone $\sigma\subseteq N_{\mathbb R}$ of dimension $n$. As mentioned at the conclusion of subsection 1.1, $d$ and $n$ are defining characteristics of $X_{\sigma}$.

\begin{Definition}
	Let $(d,n)\in\mathbb N\times\mathbb N$ with $n\leqslant d$. A \textbf{toric local model of dimension (d,n)} is a nonsingular $d$-dimensional affine toric variety $X_{\sigma}=\mathrm{Spec}\,\mathcal K[S_{\sigma}]$ constructed from a lattice $N\cong\mathbb Z^d$, and a regular strongly convex rational polyhedral cone $\sigma\subseteq N_{\mathbb R}$ of dimension $n$. We will write $\dim X_{\sigma}=(d,n)$.
\end{Definition}

\subsection{Toric Morphisms}
Our next task is to describe toric morphisms of local models precisely. Let $\mathfrak X_1$ and $\mathfrak X_2$ be toric varieties. A morphism $\phi:\mathfrak X_1\to\mathfrak X_2$ of algebraic varieties is \textbf{\textbf{\textit{toric}}} if it maps the torus $T_1\subset \mathfrak X_1$ into $T_2\subset \mathfrak X_2$ and $\phi|_{T_1}$ is a group homomorphism. When we have affine toric varieties $X_{\sigma}=\mathrm{Spec}\,\mathcal K[S_{\sigma}]$ and $Y_{\tau}=\mathrm{Spec}\,\mathcal K[S_{\tau}]$, a morphism $\phi:X_{\sigma}\to Y_{\tau}$ is toric if and only if the corresponding $\mathcal K$-algebra homomorphism $\phi^*:\mathcal K[S_{\tau}]\to\mathcal K[S_{\sigma}]$ of coordinate rings is induced by a semigroup homomorphism $S_{\tau}\to S_{\sigma}$  \cite[Definition 1.3.13 and Proposition 1.3.14]{CLS}.\\

Suppose that $\phi:X_{\sigma}\to Y_{\tau}$ is a dominant toric morphism of local models with $\dim X_{\sigma}=(d,n)$ and $\dim Y_{\tau}=(m,\ell)$. So $\phi^*:\mathcal K[S_{\tau}]\to\mathcal K[S_{\sigma}]$ is induced by a semigroup homomorphism say $\gamma:S_{\tau}\to S_{\sigma}$.
Let $\mathfrak{G_{\tau}}=\{e^*_1,\dots,e^*_{\ell},\pm e^*_{\ell+1},\dots,\pm e^*_m\}$ and $\mathfrak{G_{\sigma}}=\{e^*_1,\dots,e^*_n,\pm e^*_{n+1},\dots,\pm e^*_d\}$ be the minimal generating sets of $S_{\tau}$ and $S_{\sigma}$ respectively. Then for each $e_i^*\in\mathfrak G_{\tau}$, there exist $a_{i1},\dots,a_{in}\in\mathbb N$ and $a_{i(n+1)},\dots,a_{id}\in\mathbb Z$ such that
$$
\gamma(e^*_i)=\sum_{j=1}^{d}a_{ij}e_j^*\in\mathbb N\mathfrak G_{\sigma}.
$$
By using the isomorphism (\ref{SGAlg}), $\gamma$ induces the $\mathcal K$-algebra homomorphism
$$
\phi^*:\mathcal K[v_1,\dots,v_m]_{v_{\ell+1}\cdots v_m}\to \mathcal K[u_1,\dots,u_d]_{u_{n+1}\cdots u_d}
$$
where $\phi^*(v_i)=u_1^{a_{i1}}\cdots u_d^{a_{id}}$, for $1\leqslant i \leqslant m$.\\

We note that $e_{\ell+1}^*,\dots,e_m^*$ are invertible in $S_{\tau}$, so $\gamma(e_i^*)$ are invertible in $S_{\sigma}$, for $\ell+1\leqslant i\leqslant m$, and $\gamma(-e^*_i)=-\gamma(e^*_i)\in S_{\sigma}$. This imposes $-a_{ij}\in\mathbb N$ and hence $a_{ij}=0$, for $\ell+1\leqslant i\leqslant m$ and $1\leqslant j\leqslant n$. Thus $\phi^*:\mathcal K[S_{\tau}]\to\mathcal K[S_{\sigma}]$ is given by a system of monomial equations
\begin{equation}\label{A_1}
\begin{array}{lcccclcl}
\phi^*(v_1)                   & = & u_1^{a_{11}} & \cdots & u_n^{a_{1n}} & u_{n+1}^{a_{1(n+1)}} & \cdots & u_d^{a_{1d}}  \\
\vdots                &   &  &  &  &  &  &  \\
\phi^*(v_{\ell})              & = & u_1^{a_{\ell 1}} & \cdots & u_n^{a_{\ell n}} & u_{n+1}^{a_{\ell(n+1)}} & \cdots & u_d^{a_{\ell d}} \\
\phi^*(v_{\ell+1}) & = &  &  &  & u_{n+1}^{a_{(\ell+1)(n+1)}} & \cdots & u_d^{a_{(\ell+1)d}} \\
\vdots                &   &  &  &  &  &  &  \\
\phi^*(v_m)        & = &  &  &  & u_{n+1}^{a_{m(n+1)}} & \cdots & u_d^{a_{md}}
\end{array}
\end{equation}
where $(a_{ij})_{\substack{1\leqslant i\leqslant\ell \\1\leqslant j\leqslant n}}\in\mathbb N^{\ell\times n}$, and $(a_{ij})_{\substack{1\leqslant i\leqslant m \\ n+1\leqslant j\leqslant d}}\in\mathbb Z^{m\times (d-n)}$. In addition, since $\phi$ is dominant,
\begin{equation}\label{rkCondition}
\mathrm{rank}\,(a_{ij})_{m\times d}=\mathrm{rank}\,\left(
\begin{array}{c|c}
(a_{ij})_{\substack{1\leqslant i\leqslant\ell \\1\leqslant j\leqslant n}} &  \\
\cdots\cdots\cdots\cdots & (a_{ij})_{\substack{1\leqslant i\leqslant m \\ n+1\leqslant j\leqslant d}} \\
\mathrm{O}_{(m-\ell)\times n} &  \\
\end{array}
\right)=m.
\end{equation}
We note that $u_1,\dots,u_d$ and $v_1,\dots,v_m$ are uniformizing parameters on $X_{\sigma}$ and $Y_{\tau}$ respectively.\\

Finally, denoting by $T_1$ and $T_2$ the tori in $X_{\sigma}$ and $Y_{\tau}$ respectively, we have $\phi(T_1)=T_2$ and hence $X_{\sigma}\setminus T_1=\phi^{-1}(Y_{\tau}\setminus T_2)$ since $\phi:X_{\sigma}\to Y_{\tau}$ is a dominant toric morphism. We notice that $Y_{\tau}\setminus T_2$ is a closed subset, and by \cite[Proposition 2.27, page 17]{C4}, $\phi^{-1}(Y_{\tau}\setminus T_2)$ is the zero locus $\mathcal Z(\phi^*(\mathcal I_{Y_{\tau}\setminus T_2}))$ where $\mathcal I_{Y_{\tau}\setminus T_2}=\langle v_1\cdots v_{\ell}\rangle$ is the ideal of $Y_{\tau}\setminus T_2$ in $\mathcal K[S_{\sigma}]$. Thus by using (\ref{A_3}), we have
$$
X_{\sigma}\setminus T_1=\phi^{-1}(Y_{\tau}\setminus T_2)=\mathcal Z(\langle u_1^{\sum_{i=1}^{\ell}a_{i1}} \cdots u_n^{\sum_{i=1}^{\ell}a_{in}}\rangle)
$$
which is true if and only if for all $1\leqslant j \leqslant n$, $\sum_{i=1}^{\ell}a_{ij}\neq 0$ since $\mathcal I_{X_{\sigma}\setminus T_1}=\langle u_1\cdots u_n\rangle$.

\subsection{Toroidal Forms}
We are now ready to prove the following characterization of toroidal morphisms, the main result of this section.

\begin{Theorem}\label{TF}
	Suppose that $\varphi:X\to Y$ is a dominant morphism of nonsingular $\mathcal K$--varieties where $\dim X=d$ and $\dim Y=m$. Further suppose that there is a simple normal crossings (SNC) divisor $D_Y$ on $Y$ such that  $D_X=\varphi^{-1}(D_Y)$ is a SNC divisor on $X$ which contains the non smooth locus of
	the map $\varphi$. %$\mathbf{\rm{Sing}}\,\varphi\subset D_X$.
	Then the morphism $\varphi$ is toroidal if and only if for each $n$--point $p\in D_X$ and $\l$--point $q=\varphi(p)\in D_Y$, there exist permissible parameters\footnote{See Definition \ref{per parameters}.} $\mathbf{x}=(x_1,\dots,x_d)$ at $p$ for $D_X$, and  permissible parameters $\mathbf{y}=(y_1,\dots,y_m)$ at $q$ for $D_Y$, and there exist $(a_{ij})_{\substack{1\leqslant i\leqslant \ell \\ 1\leqslant j\leqslant n}}\in\mathbb N^{\ell\times n}$ such that $\varphi$ is given by a system of equations of the following form:
	\begin{equation}\label{TFeq}
	y_i=\left\{
	\begin{array}{ll}
	x_1^{a_{i1}} \cdots  x_n^{a_{in}}, & 1\leqslant i\leqslant r \\
	x_1^{a_{i1}} \cdots  x_n^{a_{in}}(x_{n-r+i}+\alpha_{n-r+i}), & r < i\leqslant \ell \\
	x_{n-r+i}, & \ell < i\leqslant m
	\end{array}
	\right.
	\end{equation} 
	where $r=\mathrm{rank\,}(a_{ij})_{\substack{1\leqslant i\leqslant \ell \\ 1\leqslant j\leqslant n}}=\mathrm{rank\,}(a_{ij})_{\substack{1\leqslant i\leqslant r \\ 1\leqslant j\leqslant r}}$, and $\alpha_{n+1},\dots,\alpha_{n+\ell-r}\in\mathcal K^{\times}$. In addition, 
	\begin{equation}
	\mbox{for all }j\in [n], \sum_{i=1}^{\ell}a_{ij}> 0,\mbox{ and for all }i\in[\ell], \sum_{j=1}^{n}a_{ij}> 0.
	\end{equation}
\end{Theorem}

\begin{proof}
	Suppose that dominant morphism $\varphi:(X,D_X)\to (Y,D_Y)$ is toroidal at $p\in D_X$. By definition, there exist local models $(X_{\sigma},p')$ for $X$ at $p$ and $(Y_{\tau},q')$ for $Y$ at $q$, and a toric morphism $\phi:X_{\sigma}\to Y_{\tau}$ such that the following diagram commutes:
	\begin{equation}\label{proofCD}
	\begin{CD}
	{\hat{\mathcal{O}}_{X,p}}@>\iota_p>>\hat{\mathcal{O}}_{X_{\sigma},p'}\\
	@A\hat{\varphi}^{*}AA @AA\hat{\phi}^{*}A\\
	{\hat{\mathcal{O}}_{Y,q}}@>\iota_q>>\hat{\mathcal{O}}_{Y_{\tau},q'}.
	\end{CD}
	\end{equation}
	
	If $p\in D_X$ is an $n$-point, $1\leqslant n\leqslant d$, and $q=\varphi(p)$ is an $\ell$-point for $D_Y$, $1\leqslant\ell\leqslant m$, then by Proposition \ref{SNC div Specifies a Tor structure}, $\dim X_{\sigma}=(d,n)$ and $\dim Y_{\tau}=(m,\ell)$. So there exist uniformizing parameters $u_1,\dots,u_d$ and $v_1,\dots,v_m$ on $X_{\sigma}$ and $Y_{\tau}$ respectively, such that $\mathcal O_{X_{\sigma}}=\mathcal K[u_1,\dots,u_d]_{u_{n+1}\cdots u_d}$ and $\mathcal O_{Y_{\tau}}=\mathcal K[v_1,\dots,v_m]_{v_{\ell+1}\cdots v_m}$, and there exist $(a_{ij})_{\substack{1\leqslant i\leqslant\ell \\1\leqslant j\leqslant n}}\in\mathbb N^{\ell\times n}$ and $(a_{ij})_{\substack{1\leqslant i\leqslant m \\ n+1\leqslant j\leqslant d}}\in\mathbb Z^{m\times (d-n)}$ satisfying the rank condition (\ref{rkCondition}) such that $\phi^*:\mathcal O_{Y_{\tau}}\to\mathcal O_{X_{\sigma}}$ is given by the system of monomial equations (\ref{A_1}).
	
	In addition, by the argument of Proposition \ref{SNC div Specifies a Tor structure}, there exist $\alpha_{n+1},\dots,\alpha_d\in\mathcal K^{\times}$ such that the point $p'\in X_{\sigma}$ corresponds to the maximal ideal $\langle u_1,\dots,u_n,u_{n+1}-\alpha_{n+1},\dots,u_d-\alpha_d\rangle$ of $\mathcal O_{X_{\sigma}}$. For $n+1\leqslant j \leqslant d$, we set $\tilde u_j=u_j-\alpha_j$, and by \cite[Proposition 14.18, page 171]{C4}, we obtain regular system of parameters
	$$
	\mathbf{u}_0=(u_1,\dots,u_n,\tilde u_{n+1},\dots,\tilde u_d)\in\mathcal O_{X_{\sigma},p'}
	$$
	at $p'$ such that, after substituting $\tilde u_{n+1},\dots,\tilde u_d$ for $u_{n+1},\dots,u_n$ in (\ref{A_1}), we have
	\begin{equation}\label{A' 1}
	\begin{array}{lcccclcl}
	\phi^*(v_1)                   & = & u_1^{a_{11}} & \cdots & u_n^{a_{1n}} & (\tilde u_{n+1}+\alpha_{n+1})^{a_{1(n+1)}} & \cdots & (\tilde u_d+\alpha_d)^{a_{1d}}  \\
	\vdots                &   &  &  &  &  &  &  \\
	\phi^*(v_{\ell})              & = & u_1^{a_{\ell 1}} & \cdots & u_n^{a_{\ell n}} & (\tilde u_{n+1}+\alpha_{n+1})^{a_{\ell(n+1)}} & \cdots & (\tilde u_d+\alpha_d)^{a_{\ell d}} \\
	\phi^*(v_{\ell+1}) & = &  &  &  & (\tilde u_{n+1}+\alpha_{n+1})^{a_{(\ell+1)(n+1)}} & \cdots & (\tilde u_d+\alpha_d)^{a_{(\ell+1)d}} \\
	\vdots                &   &  &  &  &  &  &  \\
	\phi^*(v_m)        & = &  &  &  & (\tilde u_{n+1}+\alpha_{n+1})^{a_{m(n+1)}} & \cdots & (\tilde u_d+\alpha_d)^{a_{md}}.
	\end{array}
	\end{equation}
	
	We are going to simplify the system (\ref{A' 1}) by changing parameters as much as possible.
	
	{\textbf{Step 1.}}Suppose that $\mathrm{rank~}(a_{ij})_{\substack{1\leqslant i\leqslant \ell\\1\leqslant j\leqslant n}}=r\leqslant\min\{\ell,n\}$. After possibly permuting variables $u_1,\dots,u_n$ and $v_1,\dots,v_{\ell}$, we may assume
	%$$
	%\mathrm{rank~}(a_{ij})_{\substack{1\leqslant i\leqslant \ell\\1\leqslant j\leqslant n}}=\mathrm{rank~}(a_{ij})_{\substack{1\leqslant i\leqslant r\\1\leqslant j\leqslant r}}=r,
	%$$and
	$\det(a_{ij})_{\substack{1\leqslant i\leqslant r\\1\leqslant j\leqslant r}}\neq 0$. Then there exists an $r\times (d-n)$ matrix $(b_{ij})\in\mathbb Q^{r\times (d-n)}$ such that
	%(with rational entries)
	$$
	(a_{ij})_{\substack{1\leqslant i\leqslant r \\ n+1\leqslant j\leqslant d}}-(a_{ij})_{\substack{1\leqslant i\leqslant r\\1\leqslant j\leqslant r}}(b_{ij})_{\substack{1\leqslant i\leqslant r\\ n+1\leqslant j\leqslant d}}=0.
	$$
	
	Thus, for $1\leqslant j\leqslant r $, there exist $\bar u_j\in\hat{\mathcal O}_{X_{\sigma},p'}$, defined by
	$$
	\bar u_j=u_j(\tilde u_{n+1}+\alpha_{n+1})^{b_{j(n+1)}}  \cdots (\tilde u_d+\alpha_d)^{b_{jd}},
	$$
	so that for $1\leqslant k\leqslant n$, we have
	$$
	\frac{\partial\bar u_j}{\partial u_k}=\left\{
	\begin{array}{ll}
	(\tilde u_{n+1}+\alpha_{n+1})^{b_{j(n+1)}}  \cdots (\tilde u_d+\alpha_d)^{b_{jd}} & k=j ; \\
	0 & k\neq j ;
	\end{array}
	\right.
	$$
	and for $n+1\leqslant k\leqslant d$,
	$$\frac{\partial\bar u_j}{\partial \tilde u_k}= b_{jk}u_j (\tilde u_k+\alpha_k)^{-1}(\tilde u_{n+1}+\alpha_{n+1})^{b_{j(n+1)}} \cdots (\tilde u_d+\alpha_d)^{b_{jd}}.$$
	
	So, writing $\mathbf{u}_1=(\bar u_1,\dots,\bar u_r,u_{r+1},\dots,u_n,\tilde u_{n+1},\dots,\tilde u_d)$, and denoting the Jacobian matrix
	$$
	J(\bar u_1,\dots,\bar u_r,u_{r+1},\dots,u_n,\tilde u_{n+1},\dots,\tilde u_d;u_1,\dots,u_n,\tilde u_{n+1},\dots,\tilde u_d)
	$$
	by $J(\mathbf{u}_1;\mathbf{u}_0)$, we have
	\begin{align*}
	\det J(\mathbf{u}_1;\mathbf{u}_0)&= \det \left(
	\begin{array}{c|c|c}
	\left(\frac{\partial\bar u_j}{\partial u_k}\right)_{\substack{1\leqslant j\leqslant r \\ 1\leqslant k\leqslant r}} & \mathrm{O} & \left(\frac{\partial\bar u_j}{\partial \tilde u_k}\right)_{\substack{1\leqslant j\leqslant r \\ n+1\leqslant k\leqslant d}} \\
	\hdotsfor[.5]{3}\\
	\mathrm{O} & I_{n-r} & \mathrm{O} \\
	\hdotsfor[.5]{3}\\
	\mathrm{O} & \mathrm{O} & I_{d-n} \\
	\end{array}
	\right)\\
	&= (\tilde u_{n+1}+\alpha_{n+1})^{\sum_{i=1}^r b_{i(n+1)}}\cdots (\tilde u_d+\alpha_d)^{\sum_{i=1}^r b_{id}}
	\end{align*}
	which is a unit in $\hat{\mathcal O}_{X_{\sigma},p'}$. Therefore, $\mathbf{u}_1=(\bar u_1,\dots,\bar u_r,u_{r+1},\dots,u_n,\tilde u_{n+1},\dots,\tilde u_d)$
	is a regular system of parameters in $\hat{\mathcal O}_{X_{\sigma},p'}$.% by Remark \ref{Remark2}.
	Substituting $\bar u_1,\dots,\bar u_r$ for $u_1,\dots,u_r$ in (\ref{A' 1}), we obtain
		\begin{equation}\label{A_2}
	\phi^*(v_i)=\left\{
	\begin{array}{ll}
	\prod_{j=1}^r\bar u_j^{a_{ij}} \times \prod_{j=r+1}^n u_j^{a_{ij}}, & 1\leqslant i\leqslant r \\
	\prod_{j=1}^r\bar u_j^{a_{ij}} \times \prod_{j=r+1}^n u_j^{a_{ij}} \times \prod_{j=n+1}^d(\tilde u_j+\alpha_j)^{c_{ij}} , & r < i\leqslant \ell \\
	\prod_{j=n+1}^d(\tilde u_j+\alpha_j)^{c_{ij}}, & \ell < i\leqslant m
	\end{array}
	\right.
	\end{equation} 
	where
	$$
	(c_{ij})_{\substack{r+1\leqslant i\leqslant m \\ n+1\leqslant j\leqslant d}}=(a_{ij})_{\substack{r+1\leqslant i\leqslant m \\ n+1\leqslant j\leqslant d}}-
	\left(
	\begin{array}{c}
	(a_{ij})_{\substack{r+1\leqslant i\leqslant \ell \\1\leqslant j\leqslant r}} \\
	\hdotsfor[.5]{1}\\
	\textrm{O}_{(m-\ell)\times r} \\
	\end{array}
	\right)(b_{ij})_{\substack{1\leqslant i\leqslant r\\ n+1\leqslant j\leqslant d}}.
	$$
	We note that $\phi^*(v_{\ell+1}),\dots,\phi^*(v_m)$ have, in fact, remained unchanged. In addition,
	$$
	\mathrm{rank\,}\left(
	\begin{array}{c|c}
	(a_{ij})_{\substack{1\leqslant i\leqslant r\\1\leqslant j\leqslant n}} & \mathrm{O}_{r\times (n-d)} \\
	\hdotsfor[.5]{2}\\
	(a_{ij})_{\substack{r+1\leqslant i\leqslant\ell \\1\leqslant j\leqslant n}} &  \\
	\hdotsfor[.5]{1} & (c_{ij})_{\substack{r+1\leqslant i\leqslant m \\ n+1\leqslant j\leqslant d}} \\
	\mathrm{O}_{(m-\ell)\times n} &  \\
	\end{array}
	\right)=\mathrm{rank\,}(a_{ij})_{\substack{1\leqslant i\leqslant m \\ 1\leqslant j\leqslant d}}=m,
	$$
	since this matrix is obtained from $(a_{ij})_{\substack{1\leqslant i\leqslant m \\ 1\leqslant j\leqslant d}}$ by performing elementary column operations. Hence we must have $\mathrm{rank~}(c_{ij})_{\substack{r+1\leqslant i\leqslant m \\ n+1\leqslant j\leqslant d}}=m-r$ because $\mathrm{rank~}(a_{ij})_{\substack{1\leqslant i\leqslant \ell\\1\leqslant j\leqslant n}}=\mathrm{rank~}(a_{ij})_{\substack{1\leqslant i\leqslant r\\1\leqslant j\leqslant r}}=r$.
	
	{\textbf{Step 2.}} For $1\leqslant s \leqslant m-r$, we set
	\begin{align}\label{step2}
	\bar \alpha_{n+s} &=\alpha_{n+1}^{c_{(r+s)(n+1)}}\cdots \alpha_d^{c_{(r+s)d}},\\
	\bar u_{n+s}      &=(\tilde u_{n+1}+\alpha_{n+1})^{c_{(r+s)(n+1)}}\cdots (\tilde u_d+\alpha_d)^{c_{(r+s)d}}- \bar \alpha_{n+s},
	\end{align}
	and hence we have
	$$
	\begin{array}{ll}
	\frac{\partial\bar u_{n+s}}{\partial\bar u_j}=0 , & \mbox{ for }1\leqslant j\leqslant r, \\
	\frac{\partial\bar u_{n+s}}{\partial u_j}    =0 , & \mbox{ for }r+1\leqslant j\leqslant n, \\
	\frac{\partial\bar u_{n+s}}{\partial \tilde u_j}=c_{(r+s)j}(\tilde u_j+\alpha_j)^{-1}(\tilde u_{n+1}+\alpha_{n+1})^{c_{(r+s)(n+1)}} \cdots (\tilde u_d+\alpha_d)^{c_{(r+s)d}}, & \mbox{ for }n+1\leqslant j\leqslant d.
	\end{array}
	$$
	We observe that $\left(\frac{\partial\bar u_{n+s}}{\partial \tilde u_{j}}\right)_{\substack{1\leqslant s\leqslant m-r \\ n+1\leqslant j\leqslant d}}=M_1(c_{ij})_{\substack{r+1\leqslant i\leqslant m \\ n+1\leqslant j\leqslant d}} M_2$ where
	\begin{align*}
	M_1 &=\left(\begin{array}{ccc}
	\prod_{j=n+1}^{d} (\tilde u_j+\alpha_j)^{c_{(r+1)j}}&        & \mathrm{O} \\
	& \ddots &  \\
	\mathrm{O} &        &\prod_{j=n+1}^{d} (\tilde u_j+\alpha_j)^{c_{mj}} \\
	\end{array}\right)_{(m-r)\times(m-r)} \mbox{ and }\\
	M_2 &=\left(\begin{array}{ccc}
	(\tilde u_{n+1}+\alpha_{n+1})^{-1} &        & \mathrm{O} \\
	& \ddots &  \\
	\mathrm{O} &        & (\tilde u_d+\alpha_d)^{-1} \\
	\end{array}\right)_{(d-n)\times (d-n)}.
	\end{align*}
	So $\mathrm{rank\,}\left(\frac{\partial\bar u_{n+s}}{\partial \tilde u_j}\right)_{\substack{1\leqslant s\leqslant m-r \\ n+1\leqslant j\leqslant d}}=\mathrm{rank\,}(c_{ij})_{\substack{r+1\leqslant i\leqslant m \\ n+1\leqslant j\leqslant d}}=m-r$. In particular, after possibly permuting variables $\tilde u_{n+1},\dots,\tilde u_d$, we may assume
	$$
	\mathrm{rank\,}(c_{ij})_{\substack{r+1\leqslant i\leqslant m \\ n+1\leqslant j\leqslant d}}=\mathrm{rank~}(c_{ij})_{\substack{r+1\leqslant i\leqslant m \\ n+1\leqslant j\leqslant n+m-r}}=m-r,
	$$
	and hence
	$$
	\mathrm{rank\,}\left(\frac{\partial\bar u_{n+s}}{\partial \tilde u_j}\right)_{\substack{1\leqslant s\leqslant m-r \\ n+1\leqslant j\leqslant d}}=\mathrm{rank\,}\left(\frac{\partial\bar u_{n+s}}{\partial \tilde u_j}\right)_{\substack{1\leqslant s\leqslant m-r \\ n+1\leqslant j\leqslant n+m-r}}=m-r.
	$$
	Thus, writing
	$$
	\mathbf{u}_2=(\bar u_1,\dots,\bar u_r,u_{r+1},\dots,u_n,\bar u_{n+1},\dots,\bar u_{n+m-r},\tilde u_{n+m-r+1},\dots,\tilde u_d),
	$$
	the Jacobian matrix
	$J(\mathbf{u}_2;\mathbf{u}_1)$,
	which is precisely
	$$
	J(\mathbf{u}_2;\mathbf{u}_1)=\left(
	\begin{array}{c|c|c}
	I_n & \mathrm{O} & \mathrm{O} \\
	\hdotsfor[.5]{3}\\
	\mathrm{O} & \left(\frac{\partial\bar u_{n+s}}{\partial\bar u_j}\right)_{\substack{1\leqslant s\leqslant m-r \\ n+1\leqslant j\leqslant n+m-r}} &
	\left(\frac{\partial\bar u_{n+s}}{\partial\tilde u_j}\right)_{\substack{1\leqslant s\leqslant m-r \\ n+m-r+1\leqslant j\leqslant d}} \\
	\hdotsfor[.5]{3}\\
	\mathrm{O} & \mathrm{O} & I_{d-(n+m-r)} \\
	\end{array}
	\right),$$
	and after performing elementary row operations becomes
	$$J(\mathbf{u}_2;\mathbf{u}_1)\leftrightarrow\left(
	\begin{array}{c|c|c}
	I_n & \mathrm{O} & \mathrm{O} \\
	\hdotsfor[.5]{3}\\
	\mathrm{O} & \left(\frac{\partial\bar u_{n+s}}{\partial\bar u_j}\right)_{\substack{1\leqslant s\leqslant m-r \\ n+1\leqslant j\leqslant n+m-r}} & \mathrm{O} \\
	\hdotsfor[.5]{3}\\
	\mathrm{O} & \mathrm{O} & I_{d-(n+m-r)} \\
	\end{array}
	\right),
	$$
	has rank $d$. Therefore, $\mathbf{u}_2$ is a regular system of parameters in $\hat{\mathcal O}_{X_{\sigma},p'}$.% by Remark \ref{Remark2}.
	We note $\bar\alpha_{n+1},\dots,\bar\alpha_{n+m-r}$, defined in (\ref{step2}), belong to $\mathcal K^{\times}$, i.e., are nonzero constant. Substituting $\bar u_{n+1},\dots,\bar u_{n+m-r}$ and $\bar\alpha_{n+1},\dots,\bar\alpha_{n+m-r}$ in (\ref{A_2}), we obtain
	\begin{equation}\label{A_3}
	\phi^*(v_i)=\left\{
	\begin{array}{ll}
	\prod_{j=1}^r\bar u_j^{a_{ij}} \times \prod_{j=r+1}^n u_j^{a_{ij}}, & 1\leqslant i\leqslant r \\
	\prod_{j=1}^r\bar u_j^{a_{ij}} \times \prod_{j=r+1}^n u_j^{a_{ij}} \times (\bar u_{n-r+i}+\bar\alpha_{n-r+i}) , & r < i\leqslant \ell \\
	\bar u_{n-r+i}+\bar\alpha_{n-r+i}, & \ell < i\leqslant m.
	\end{array}
	\right.
	\end{equation} 
	Finally, if, for $\ell+1 \leqslant i \leqslant m$, we set $\bar v_i=v_i-\bar\alpha_{n-r+i}$, we obtain regular system of parameters
	$$
	\mathbf{v}_0=(v_1,\dots,v_{\ell},\bar v_{\ell+1},\dots,\bar v_m)\in \mathcal O_{Y_{\tau},q'}
	$$
	at $q'=\phi(p')$ such that
	\begin{equation}\label{A_4} 
	\left\{
	\begin{array}{ll}
	\phi^*(v_i)=\prod_{j=1}^r\bar u_j^{a_{ij}} \times \prod_{j=r+1}^n u_j^{a_{ij}}, & 1\leqslant i\leqslant r \\
	\phi^*(v_i)=\prod_{j=1}^r\bar u_j^{a_{ij}} \times \prod_{j=r+1}^n u_j^{a_{ij}} \times (\bar u_{n-r+i}+\bar\alpha_{n-r+i}) , & r < i\leqslant \ell \\
	\phi^*(\bar v_i)=\bar u_{n-r+i} , & \ell < i\leqslant m.
	\end{array}
	\right.
	\end{equation}
	
	Clearly, regular systems of parameters $\mathbf{x}=(x_1,\dots,x_d)$ and $\mathbf{y}=(y_1,\dots,y_m)$, defined by
	$$
	x_j:=\left\{
	\begin{array}{ll}
	\iota_p^{-1}(\bar u_j), & \hbox{ for } 1 \leqslant j \leqslant r, \\
	\iota_p^{-1}(u_j), & \hbox{ for } r+1 \leqslant j \leqslant n, \\
	\iota_p^{-1}(\bar u_j), & \hbox{ for } n+1 \leqslant j \leqslant n+m-r, \\
	\iota_p^{-1}(\tilde u_j), & \hbox{ for } n+m-r+1 \leqslant j \leqslant d,
	\end{array}
	\right.
	\mbox{ and }
	y_i:=\left\{
	\begin{array}{ll}
	\iota_q^{-1}(v_i), & \hbox{ for } 1 \leqslant i \leqslant \ell, \\
	\iota_q^{-1}(\bar v_i), & \hbox{ for } \ell+1 \leqslant i \leqslant m, \\
	\end{array}
	\right.
	$$
	are permissible parameters at $p$ and $q=\varphi(p)$ respectively, such that the conclusion holds.
\end{proof}

\section{Resolution Techniques}\label{Resolution Techniques}
In this section we study some resolution tools that we need to prove 
%the existence of a 
toroidalization of locally toroidal morphisms. Given a locally toroidal morphism $\varphi:(X,U_{\a},D_{\a})_{\a\in I}\to (Y,V_{\a},E_{\a})_{\a\in I}$, first we need to construct a SNC divisor as a toroidal structure on $Y$ containing all $E_{\a}$. This can be done using the algorithm of embedded resolution of singularities. We also need to apply a specific algorithm for principalization of an ideal sheaves. Among all the proofs of principalization and resolution in characteristic zero such as \cite{BM, BEV, EH, EV, V, W}, which are simplifications of the Hironaka's original proofs \cite{H}, we will use the algorithm of \cite{BEV,EV}. This is explained nicely in the book by Cutkosky \cite{C2} which is our primary reference.\\

We start by recalling the construction of blowing-up and other preliminaries we need to understand the embedded resolution algorithm. We then prove in Theorem \ref{ERS} some geometric facts which are essential in constructing toroidalization of locally toroidal morphisms. We also review some properties of the strong principalization algorithm \cite{C2}[Theorem 6.35], and we will study the effect of a principalization sequence on toroidal forms in Lemma \ref{QTF}.

\subsection{Blowing-up and other preliminaries}
\subsubsection{Blowing-up}
Suppose that $X$ is a nonsingular $d$--dimensional variety. Let $\pi:\widetilde X \to X$ be the blowup of $X$ along a nonsingular irreducible subvariety $T\subset X$ with $\mathrm{codim}_XT=c$. There exist regular parameters $x_1,\dots,x_d\in\mathcal O_{X,p}$ at $p\in T$, and an affine open neighborhood $U$ of $p$ in $X$ such that $x_1=\cdots=x_d=0$ are local equations of $p$ in $U$, and $x_1=\cdots=x_c=0$ are local equations of $T$ in $U$, i.e., $\mathfrak m_p=(x_1,\dots,x_d)$ is the maximal ideal of $R=\Gamma(U,\mathcal O_X)$, and $I=\Gamma(U,\mathcal I_T)=(x_1,\dots,x_c)$. Then $\pi^{-1}(U)$ has the affine cover
$$
\pi^{-1}(U)=\mathrm{Proj}\ (\bigoplus_{n=0}^{\infty} I^n)=\bigcup_{i=1}^c \mathrm{Spec}\ R[\frac{x_1}{x_i},\dots,\frac{x_c}{x_i}].
$$	
Suppose that $\tilde p\in\pi^{-1}(p)$. Then $\tilde p$ lies in $\mathrm{Spec}\ R[\frac{x_1}{x_i},\dots,\frac{x_c}{x_i}]$ for some $1\leqslant i \leqslant c$.
In \cite [Theorem 10.19]{C4}, it is shown that there exist $\alpha_1,\dots,\alpha_c\in\mathcal K$ such that
$$
\frac{x_1}{x_i}-\alpha_1,\dots,\widehat{\frac{x_i}{x_i}-\alpha_{i}},\dots,\frac{x_c}{x_i}-\alpha_c,x_i,x_{c+1},\dots,x_d
$$
generate the maximal ideal $\mathfrak m_{\tilde p}$ of $\mathcal O_{\widetilde X,\tilde p}$\ , and $\dim\mathfrak m_{\tilde p}/\mathfrak m_{\tilde p}^2=\dim\mathcal O_{\widetilde X,\tilde p}=d$. (Here the hat denotes as usual an element omitted from the list). Thus $\mathcal O_{\widetilde X,\tilde p}$ is a regular local ring. Setting
$$
\tilde x_1=\frac{x_1}{x_i}-\alpha_1,\dots,\tilde x_i=x_i,\dots,\tilde x_c=\frac{x_c}{x_i}-\alpha_r,\tilde x_{c+1}=x_{c+1},\dots,\tilde x_d=x_d,
$$
we obtain regular parameters $\tilde x_1,\dots,\tilde x_d\in\mathcal O_{\widetilde X,\tilde p}$ such that the point $\tilde p$ can be described by the equations
\begin{equation}\label{blup eqs}
x_j=\left\{
\begin{array}{ll}
\tilde x_i(\tilde x_j+\alpha_j) & \hbox{for $1\leqslant j\leqslant c$ and $j\neq i$;} \\
\tilde x_i & \hbox{for $j=i$;} \\
\tilde x_j & \hbox{for $j>c$.}
\end{array}
\right.
\end{equation}
Furthermore, we have that
$$
\mathcal I_T\mathcal O_{\widetilde X,\tilde p}=(x_1,\dots,x_c)\mathcal O_{\widetilde X,\tilde p}=(\tilde x_i(\tilde x_1+\alpha_1),\dots,\tilde x_i,\dots,\tilde x_i(\tilde x_c+\alpha_c))=\tilde x_i\mathcal O_{\widetilde X,\tilde p}.
$$
So $\tilde x_i=0$ is a local equation of $\pi^{-1}(T\cap U)\cong (T\cap U)\times\mathbb P^{c-1}$, the exceptional divisor of $\pi$, in $\mathrm{Spec}\ R[\frac{x_1}{x_i},\dots,\frac{x_c}{x_i}]$. Thus $\widetilde X$ is non-singular at $\tilde p\in\pi^{-1}(T)$, and $\mathcal I_T\mathcal O_{\widetilde X,\tilde p}$ is principal. In addition, $\widetilde X\setminus\pi^{-1}(T)\to X\setminus T$ is an isomorphism and therefore $\widetilde X$ is nonsingular and $\mathcal I_T\mathcal O_{\widetilde X}$ is a locally principal ideal sheaf on $\widetilde X$.

%%%%%%%%%%%%%%%%%%%%%%%%%%%%%%%%%%%%%%%%%%%%%%%%%%%%%%%%%%%%%%%%%%%%%%%%%%%%
\subsubsection{The Resolution Invariant}

The principal invariant considered in the algorithms of strong principalization and embedded resolution of singularities of \cite{BEV,EV} is the order of ideals.

\begin{Definition}\label{ord}
	Let $(R,\mathfrak m)$ be a regular local ring containing a field of characteristic zero. Suppose that $J\subset R$ is an ideal. The \textbf{order} of $J$ in $R$ is
	$$
	\nu_R(J)=\max\{k\ |\ J\subset \mathfrak m^k\}.
	$$
	(The order of $J$ is its multiplicity if $J$ is a locally principal ideal, e.g., $J$ is the defining ideal of a hypersurface). 
\end{Definition}

The basic strategy of resolution is to built a sequence of monoidal transforms centered in the locus of points where a suitable family of upper semi-continuous functions reach their maximum value -- see \cite[Theorems 6.34, 6.35, 6.36 and 6.37]{C2}. These functions are defined inductively and the order is one of the main invariants used to construct these functions -- see \cite[Remark 5.7]{BEV}.

\begin{Remark}
	The key feature of the algorithm, which implies all our required facts, is that the locus of points where these functions reach their maximum value is contained in the closed sets of points of maximum order. For this reason, we can avoid going through the technicalities of the algorithm, and we will simplify notation as much as possible.
	% avoid worrying about  
\end{Remark}

\begin{Definition}
	Suppose that $X$ is a nonsingular variety and $\mathcal I$ is an ideal sheaf in $\mathcal O_X$. For $p\in X$,
	$$
	\nu_p(\mathcal I)=\nu_{\mathcal O_{X,p}}(\mathcal I\mathcal O_{X,p})=\nu_{\mathcal O_{X,p}}(\mathcal I_p).
	$$
	If $V\subset X$ is a subvariety, we denote $\nu_p(V)=\nu_p(\mathcal I_{V})$. For a natural number $t\in\mathbb N$, we define
	$$
	\mathbf{Sing}_t(\mathcal I)=\{p\in X\ |\ \nu_p(\mathcal I)\geqslant t\},
	$$
	which is Zariski closed in $X$, since $\nu_{\mathcal I}:X\to \mathbb N$, given by $\nu_{\mathcal I}(p)=\nu_p({\mathcal I})$ for $p\in X$, is an upper semi-continuous function \cite [Theorem A.19]{C2}. Let
	$$
	r=\max\nu_{\mathcal I}=\max\{\nu_p(\mathcal I)\mid p\in X\}.
	$$
	Then the \textbf{maximum order locus} of $\mathcal I$ in $X$, denoted by $\mathbf{Max}\ \nu_{\mathcal I}$, is the closed subset
	$$
	\mathbf{Max}\ \nu_{\mathcal I}=\mathbf{Sing}_r(\mathcal I)=\{p\in X\ |\ \nu_p(\mathcal I)=r\}.
	$$
\end{Definition}

\begin{Definition}\label{SNCs}
	Suppose that $V$ is a $d$--dimensional nonsingular $\mathcal K$--variety. 
	\begin{enumerate}
		\item Let $W$ be a subscheme of $V$, and $\mathcal I\subset\mathcal O_V$ be its defining ideal sheaf. We say that $W$ \textbf{\textit{has SNC at}} $p\in V$ if there exist regular system of parameters $v_1,\dots,v_d\in\mathcal O_{V,p}$ and $(a_{ij})_{\substack{1\leqslant i\leqslant c \\1\leqslant j\leqslant d}}\in\mathbb N^{c\times d}$ such that
		\begin{equation*}
		\mathcal I_p=\mathcal I\mathcal O_{V,p}=\langle v_1^{a_{11}}\cdots v_d^{a_{1d}}, \dots , v_1^{a_{c1}}\cdots v_d^{a_{cd}} \rangle,
		\end{equation*}
		where $c=\max\{\mathrm{codim}_V T\ |\ T \mbox{ is an irreducible component of } W \mbox{ containing }p\}$. We say that $W$ is a \textbf{\textit{SNC subscheme}} if it has SNC at all $p\in V$ \footnote{That is, algebraically, the ideal sheaf $\mathcal I$ of $W$ is locally monomial, or geometrically, all of the irreducible components of $W$ are nonsingular and they meet pairwise transversally.}, e.g., the support of a SNC divisor $E$ on $V$ is a purely 1-codimensional SNC subscheme in $V$.
		
		\item Suppose that $E$ is a SNC divisor on $V$ and let $Z\subset V$ be a nonsingular irreducible subvariety. We say that $Z$ \textbf{\textit{makes SNC with}} $E$ at $p\in V$ if $E\cup Z$ has SNC at $p$, and we say $Z$ \textbf{\textit{makes SNC with}} $E$ if $E\cup Z$ is a SNC subscheme of $V$. More precisely, $Z$ makes SNC with $E$ if for all $p\in V$, there exist a regular system of parameters $v_1,\ldots,v_d$ at $p$, natural numbers $r\ge 0$ and $a_1,\ldots,a_d\ge 0$ such that $v_1^{a_1}\cdots v_d^{a_d}=0$ is a local equation of $E$ at $p$, and there exists $\{i_1,\dots,i_r\}\subset [d]$ such that $v_{i_1}=\cdots=v_{i_r}=0$ are local equations of $Z$ at $p$. (If $p\not\in Z$, then we take $r=0$ and if $p\not\in E$, we take $a_i=0$ for all $i$).
	\end{enumerate}
\end{Definition}

\subsection{Embedded Resolution of Singularities}
We now apply the ERS to construct a global toroidal structure on a locally toroidal variety $(Y,V_{\a},E_{\a})_{\a \in I}$. In the following theorem, We give two results about the resolution sequence of $\widetilde{\mathcal{E}}_0=\sum_{\alpha \in I}\overline{E}_{\alpha}$ which are fundamental to constructing toroidalization. 
\begin{Theorem}\label{ERS}
	Let $(V_{\alpha},E_{\alpha})_{\alpha\in I}$ be local toroidal data of a nonsingular variety $Y$, and consider the hypersurface $\widetilde{\mathcal{E}}_0=\sum_{\alpha \in I}\overline{E}_{\alpha}\subset Y$, where $\overline{E}_{\alpha}$ is the Zariski closure of $E_{\alpha}$ in $Y$. There exists a finite sequence
	\begin{equation}
	\pi:\widetilde Y = Y_{n_0}\xrightarrow{\pi_{n_0}} Y_{n_0-1}\xrightarrow{}\cdots\xrightarrow{}Y_k\xrightarrow{\pi_k}Y_{k-1}\xrightarrow{}\cdots\to Y_1\xrightarrow{\pi_1}Y_0=Y
	\end{equation}
	of monoidal transforms centered in the closed sets of points of maximum order
	such that $\widetilde Y$ is nonsingular and the total transform $\pi^{-1}(\widetilde{\mathcal E}_0)$ is a SNC divisor on $\widetilde Y$, i.e., $\pi$ is
	an embedded resolution of singularities of $\widetilde{\mathcal{E}}_0$. For $0\leqslant k\leqslant n_0$ and $\alpha \in I$, let:
	\begin{center}
		\begin{tabular}{lll}
			
			$\Pi_k=\pi_0\circ\cdots\circ\pi_k$, & \ \  & $V_{k,\alpha}=\Pi_k^{-1}(V_{\alpha})$ \\
			
			$\pi_{k,\alpha}=\pi_k|_{V_{k,\alpha}}:V_{k,\alpha}\rightarrow V_{k-1,\alpha}$, & \ \  & $\Pi_{k,\alpha}=\Pi_k|_{V_{k,\alpha}}:V_{k,\alpha}\to V_{\alpha}$, \\
			
			$Z_{k}$: the center of $\pi_{k+1}$, & \ \  & $Z_{k,\alpha}=Z_k\cap V_{k,\alpha}$, \\
			
			$E_{k,\alpha}=\Pi_{k,\alpha}^{-1}(E_{\alpha})=\Pi_{k,\alpha}^*(E_\alpha)_{\rm red}$, & \ \ & $\widetilde{\mathcal E}_k=(\sum_{\alpha\in I}\overline E_{k,\alpha})_{\rm red}$, 
		\end{tabular}
	\end{center}
	where $\overline E_{k,\alpha}$ is the Zariski closure of $E_{k,\alpha}$ in $Y_k$. We further have that: 
	\begin{enumerate}
		\item\label{ERS1} $E_{k,\alpha}$ is a SNC divisor on $V_{k,\alpha}$ for all $k,\alpha$, and $Z_{k,\alpha}$ makes SNCs with
		$E_{k,\alpha}$ on $V_{k,\alpha}$ for all $k,\alpha$. (Although possibly $Z_{k,\alpha}\cap E_{k,\alpha}\ne \emptyset$ but
		$Z_{k,\alpha}\not\subset E_{k,\alpha}$).
		\item\label{ERS2} $\widetilde{\mathcal E}_k\subseteq\Pi_k^{-1}(\widetilde{\mathcal E}_0)$ for all $k$.\footnote{See \cite[Example 2.5]{A1} for an example which shows that the inclusion of (\ref{ERS2}) is not in general an equality.}
	\end{enumerate}
\end{Theorem}

This theorem is a generalization of \cite[Theorem 3.4]{A2} to higher-dimensional varieties, and we will prove it by a completely similar argument. Our proof is based on \cite[Lemmas 3.6 and 3.7]{A2}, which were proved in arbitrary dimensions. To be self-contained, we repeat them below. We just note that the argument of \cite[Lemmas 3.6]{A2} works for a slightly more general statement.

\begin{Lemma}[Lemma 3.6 \cite{A2}]\label{Lemma1}
	Suppose that $E$ is a hypersurface embedded in a nonsingular $\mathcal K$--variety $V$. Let $Z_k$ be the permissible center of a transformation $V_{k+1}\to V_{k}$, $k\geqslant 0$, in the sequence of embedded resolution of singularities of $E$ constructed in \cite[Theorem 6.37]{C2}. Then the strict transform $A_k$ of any irreducible component of $E$ on $V_k$ satisfies:
	\begin{equation}
	Z_k\subset A_k \mbox{ or }Z_k\cap A_k=\emptyset.
	\end{equation}
\end{Lemma}

\begin{Lemma}[Lemma 3.7 \cite{A2}]\label{Lemma2}
	Suppose that $V$ is a nonsingular $\mathcal K$--variety and $E=A+B$ is a SNC divisor on $V$ where $A$ and $B$ have no irreducible components in common. Suppose that $Z$ is a nonsingular subvariety of $V$ such that $Z$ makes SNCs with $A$ and if $C$ is an irreducible component of $B$ then either $Z\subset C$ or $Z\cap C=\emptyset$. Then $Z$ makes SNCs with $E$.
\end{Lemma}

\begin{proof}[Proof of Theorem \ref{ERS}]
	By \cite[Theorem 6.41]{C2} and its proof, the embedded resolution of singularities of the hypersurface $\widetilde{\mathcal{E}}_0\subset Y$, constructed in \cite[Theorem 6.37]{C2}, is a finite sequence of monoidal transforms $\pi:\widetilde Y\to Y$ centered in the maximum order locus where $\widetilde Y$ is nonsingular and the total transform $\pi^{-1}(\widetilde{\mathcal E}_0)$ is a SNC divisor on $\widetilde Y$. More precisely, denoting by $\mathcal J_0$ the defining ideal sheaf of $\widetilde{\mathcal{E}}_0$, we have a sequence 
	\begin{equation}
	\pi:\widetilde Y = Y_{n_0}\xrightarrow{\pi_{n_0}} Y_{n_0-1}\xrightarrow{}\cdots\xrightarrow{}Y_k\xrightarrow{\pi_k}Y_{k-1}\xrightarrow{}\cdots\to Y_1\xrightarrow{\pi_1}Y_0=Y
	\end{equation}
	of monoidal transforms such that each $\pi_k$ is the blowup of a nonsingular irreducible subvariety of $\mathbf{Max}\ \nu_{\bar{\mathcal J}_{k-1}}$ where $\bar{\mathcal J}_{k-1}$ is the weak transform of $\mathcal J_0$ in $Y_{k-1}$, and $\bar{\mathcal J}_0=\mathcal J_0$. We note that since $\mathcal J_0$ is a locally principal ideal sheaf, and each transformation in this sequence is centered in a nonsingular subvariety of the maximum order locus, the strict transform of $\mathcal J_0$ on $Y_k$ coincides with its weak transform $\bar{\mathcal J}_{k}$.
	
	We will prove the properties (\ref{ERS1}) and (\ref{ERS2}) by induction on the length $k$ of the resolution sequence. The base case $k=0$ follows obviously by applying Lemmas \ref{Lemma1} and \ref{Lemma2} to $E_{\a}=\emptyset+E_{\a}$ on $V_{\a}$, so for each $\a\in I$, $Z_{0,\a}=Z_0\cap V_{\a}$ makes SNCs with $E_{\a}$. Now assume the statement is true for any sequence of length strictly less than $k$, and consider a resolution sequence $Y_k\to Y$ of length $k$. First we note that for all $\a\in I$, $E_{k,\a}$ is the pullback $\pi_{k,\a}^{-1}(E_{k-1,\a})$ of $E_{k-1,\a}$ via the blowup $\pi_{k,\a}:V_{k,\a}\to V_{k-1,\a}$ with center $Z_{k-1,\a}$. By the induction hypothesis, for all $\a\in I$, $E_{k-1,\a}$ is a SNC divisor on $V_{k-1,\a}$ and $Z_{k-1,\a}$ makes SNCs with $E_{k-1,\a}$ on $V_{k-1,\a}$. Hence for all $\a\in I$, $E_{k,\a}$ is a SNC divisor on $V_{k,\a}$. We now decompose $E_{k,\a}=A+B$ where $A$ is the sum of exceptional components of $\Pi_{k,\a}$ and $B$ is the strict transform of $E_{\a}$ on $V_{k,\a}$ via $\Pi_{k,\a}$. We must show that $Z_{k,\a}=Z_k\cap V_{k,\a}$ makes SNCs with $E_{k,\a}$. Clearly, $A$ and $B$ have no irreducible components in common, and by Lemma \ref{Lemma1}, if $C$ is an irreducible component of $B$ then either $Z_{k,\a}\subset C$ or $Z_{k,\a}\cap C=\emptyset$. In addition, $Z_k$ is permissible for the resolution algorithm (see \cite[Definition 6.25]{C2}), and hence $Z_{k,\a}$ makes SNCs with $A$. Therefore, Lemma \ref{Lemma2} implies that $Z_{k,\a}$ makes SNCs with $E_{k,\a}$. Finally, we observe that (\ref{ERS2}) holds trivially just as $Z_k$ is a permissible center for the algorithm and $Z_k\subseteq\Pi_k^{-1}(\widetilde{\mathcal{E}}_0)$ for all $k$.
	%For (\ref{ERS2}), we first note that $Z_k\subseteq\Pi_k^{-1}(\widetilde{\mathcal{E}}_0)$ for all $k$ since $Z_k$ is a permissible center for the algorithm, and in particular, $Z_{k-1}\subseteq\Pi_{k-1}^{-1}(\widetilde{\mathcal{E}}_0)$.	
	% By the induction hypothesis
	%$$
	%	E_{k-1,\alpha}\subset \widetilde{\mathcal E}_{k-1}=(\sum_{\alpha\in I}\overline E_{k-1,\alpha})_{\rm red}\subset \Pi_{k-1}^{-1}(\widetilde{\mathcal E}_0)
	%	$$
\end{proof}

\begin{Definition}\label{n-SV}
	Suppose that $V$ is a nonsingular variety of dimension $m$ and $E$ is a SNC divisor on $V$. Let $Z\subset V$ be a $c$ codimensional nonsingular irreducible subvariety of $V$ which makes SNC with $E$ at $p\in V\cap Z$. We say that $Z$ is an 
	$\mathbf{(\l,\bar \l)\mbox{\textbf{--subvariety (SV) for }} E \mbox{\textbf{ at }} p}$ if $Z$ lies in exactly $\bar \l$ irreducible components of $E$ and $p$ lies in exactly $\l$ irreducible components of $E$. We have $0\leqslant \l\leqslant m$, $0\leqslant \bar \l \leqslant \mathrm{codim}_V Z=c$ and $\bar \l\leqslant \l$.
\end{Definition}

% We note that for any permissible center in the ERS, there are $\l,\bar\l$ such that ...
% Suppose q is an \l point for E_{\a} and Z is a permissible center in the ERS ...

\begin{Definition}\label{SV PerPar}	
	Suppose that $Z\subset (V,E)$ is an $(\l,\bar \l)$--SV for $E$ at $q\in V\cap Z$ with $\mathrm{codim}_V Z=c$. We say that $y_1,\dots,y_m$ are \textbf{\textit{(formal) permissible parameters for $Z$ at $q$ (on $E$)}} if $y_1,\dots,y_m$ are regular parameters in $\hat{\mathcal O}_{V,q}$ such that $y_1\cdots y_{\l}=0$ is a (formal) local equation of $E$ at $q$, and local equations of $Z$ at $q$ are
	$$
	y_1=\cdots=y_{\bar\l}=y_{\l+1}=\cdots=y_{\l+(c-\bar\l)}=0.
	$$
	(We use the convention that $\{\l+1,\dots,\l+(c-\bar\l)\}=\emptyset$ if $c=\bar\l$, so that in this case $y_{\l+1},\dots,y_{\l+(c-\bar\l)}$ are not contained in the local equations of $Z$). We say that permissible parameters $y_1,\dots,y_m$ are \textbf{\textit{algebraic}} if $y_1,\dots,y_m\in\mathcal O_{V,q}$.
\end{Definition}

\begin{Lemma}\label{tf in terms of perpars for SVs}
	Suppose that $\varphi:(U,D)\to (V,E)$ is a toroidal morphism. Let $q\in E$ be an $\ell$--point for $E$ and let $Z\subset V$ be a $c$ codimensional $(\ell,\bar\ell)$--SV for $E$ at $q\in V$. Suppose that $p\in\varphi^{-1}(q)$ is an $n$--point for $D$. There exist permissible parameters $y_1,\dots,y_m$ for $Z$ at $q$ on $E$, and permissible parameters $x_1,\dots,x_d$ at $p$ such that
	\begin{equation}\label{TF'}
	y_i=\left\{
	\begin{array}{ll}
	\delta_i x_1^{a_{i1}} \cdots  x_n^{a_{in}}, & 1\leqslant i\leqslant \ell \\
	\ \ \ x_{n-\ell+i}, & \ell+1\leqslant i\leqslant m
	\end{array}
	\right.
	\end{equation}
	where  $\delta_i \in \hat{\mathcal O}_{U,p}$, $1\leqslant i\leqslant \ell$, are units in which the variables $x_1,\dots,x_{n+(m-\l)}$ do not appear, and $(a_{ij})_{\substack{1\leqslant i\leqslant \ell \\ 1\leqslant j\leqslant n}}\in\mathbb N^{\ell\times n}$ satisfies:
	\begin{equation}\label{ex matrix condition}
	\mbox{for all }j\in [n], \sum_{i=1}^{\ell}a_{ij}> 0,\mbox{ and for all }i\in[\ell], \sum_{j=1}^{n}a_{ij}> 0.
	\end{equation}
\end{Lemma}
\begin{proof}
	Since $\varphi:(U,D)\to (V,E)$ is toroidal, Theorem \ref{TF} tells us that there exist permissible parameters $\mathbf{y}=(y_1,\dots,y_m)$ at $q$, permissible parameters $\mathbf x=(x_1,\dots,x_d)$ at $p$, an exponent matrix $(a_{ij})_{\substack{1\leqslant i\leqslant \ell \\ 1\leqslant j\leqslant n}}\in\mathbb N^{\ell\times n}$ of rank $r$, and $\alpha_{n+1},\dots,\alpha_{n-r+\ell}\in\mathcal K^{\times}$ such that 
	\begin{equation}\label{tfeq1}
	y_i=\left\{
	\begin{array}{ll}
	x_1^{a_{i1}} \cdots  x_n^{a_{in}}, & 1\leqslant i\leqslant r \\
	x_1^{a_{i1}} \cdots  x_n^{a_{in}}(x_{n-r+i}+\alpha_{n-r+i}), & r+1\leqslant i\leqslant \ell \\
	x_{n-r+i}, & \ell+1\leqslant i\leqslant m
	\end{array}
	\right.
	\end{equation}
	where $\det(a_{ij})_{\substack{1\leqslant i\leqslant r \\ 1\leqslant j\leqslant r}}\neq 0$, $\sum_{i=1}^{\ell}a_{ij}\neq 0$, for $r+1\leqslant j\leqslant n$, and $\sum_{j=1}^{n}a_{ij}\neq 0$, for $r+1\leqslant i\leqslant \ell$. After possibly reindexing the variables\footnote{See the proof of Theorem \ref{TF}.} $x_{n+1},\dots,x_{n-r+m}$, we can rewrite (\ref{tfeq1}) as
	\begin{equation}\label{tfeq2}
	y_i=\left\{
	\begin{array}{ll}
	x_1^{a_{i1}} \cdots  x_n^{a_{in}}, & 1\leqslant i\leqslant r \\
	x_1^{a_{i1}} \cdots  x_n^{a_{in}}(x_{n-r+m-\ell+i}+\alpha_{n-r+i}), & r+1\leqslant i\leqslant \ell \\
	x_{n-\ell+i}, & \ell+1\leqslant i\leqslant m.
	\end{array}
	\right.
	\end{equation}
	This can also be written as
	\begin{equation}\label{tfeq3}
	y_i=\left\{
	\begin{array}{ll}
	\delta_i x_1^{a_{i1}} \cdots  x_n^{a_{in}}, & 1\leqslant i\leqslant \ell \\
	\ \ \ x_{n-\ell+i}, & \ell+1\leqslant i\leqslant m
	\end{array}
	\right.
	\end{equation}
	where $\delta_1,\dots,\delta_{\ell}$ are units in $\hat{\mathcal O}_{X,p}$ defined by
	\begin{equation}
	\delta_i:=\prod_{t=r+1}^{\ell} (x_{n-r+m-\ell+t}+\alpha_{n-r+t})^{\epsilon_{it}}\ \ \mbox{ with }\ \ (\epsilon_{it})_{\substack{1\leqslant i\leqslant \ell \\ r+1\leqslant t\leqslant \ell}}=\left(\begin{array}{c}
	\textrm{O}_{r\times \ell-r}\\
	\hdotsfor[.5]{1}\\
	I_{\ell-r}
	\end{array}\right)\in\{0,1\}^{\ell\times \ell-r}.
	\end{equation}
	We also note that $Z$ is a $c$ codimensional $(\ell,\bar\ell)$--SV for $E$ at $q$, meaning that $Z$ makes SNCs with $E$ and it lies in exactly $\bar\ell$ irreducible components of $E$ which has local equation $y_1\cdots y_{\ell}=0$ at $q$. Clearly, it is possible to obtain permissible parameters for $Z$ at $q$ on $E$ by permuting $y_1,\dots,y_{\ell}$, and changing the parameters $y_{\ell+1},\dots,y_m$, if necessary. When we do this, we get permissible parameters $\mathbf{\tilde y}=(\tilde y_1,\dots,\tilde y_m)$ at $q$, which satisfies $\det J(\mathbf{\tilde y};\mathbf{y})\neq 0$ by the formal inverse function theorem, such that
	\begin{equation}
	\tilde y_1=\cdots=\tilde y_{\bar \ell}=\tilde y_{\ell+1}=\cdots=\tilde y_{\ell+(c-\bar \ell)}=0
	\end{equation}
	are local equations of $Z$ at $q$\footnote{If $\bar\ell=c$, local equations of $Z$ at $q$ are $\tilde y_1=\cdots=\tilde y_{\bar \ell}=0$.}. If we further set $\tilde x_{n-\ell+i} := \tilde y_i$, for $\ell+1\leqslant i\leqslant m$, we obtain
	\begin{equation}\label{tfeq4}
	\tilde y_i=\left\{
	\begin{array}{ll}
	\delta_i x_1^{a_{i1}} \cdots  x_n^{a_{in}}, & 1\leqslant i\leqslant \ell \\
	\ \ \ \tilde x_{n-\ell+i}, & \ell+1\leqslant i\leqslant m
	\end{array}
	\right.
	\end{equation}
	where $\mathrm{rank\,}(a_{ij})_{\substack{1\leqslant i\leqslant \ell \\ 1\leqslant j\leqslant n}}=r$, $\sum_{i=1}^{\ell}a_{ij}\neq 0$, for $1\leqslant j\leqslant n$, and $\sum_{j=1}^{n}a_{ij}\neq 0$, for $1\leqslant i\leqslant \ell$. To finish the proof, we need to show that $\mathbf{\tilde x}=(x_1,\dots,x_n,\tilde x_{n+1},\dots,\tilde x_{n+m-\ell},x_{n+m-\ell+1},\dots,x_d)$ is a system of permissible parameters at $p$. It is sufficient to observe that $\det J(\mathbf{\tilde x};\mathbf{x})\neq 0$. %which follows from $\det J(\mathbf{\tilde y};\mathbf{y})\neq 0$. 
	Suppose that, for $\ell+1\leqslant i\leqslant m$, 
	\begin{equation}\label{expansion of tilde y}
	\tilde y_i=\sum_{k=1}^m \gamma_{ik}y_k + \mbox{ higher degree terms in }y_k\mbox{'s}
	\end{equation}
	for some $\gamma_{ik}\in\mathcal K$. After substituting (\ref{TFeq}) in (\ref{expansion of tilde y}), for $\ell+1\leqslant i\leqslant m$, we obtain the expansion of $\tilde x_{n-\ell+i}$ in terms of $x_1,\dots, x_d$ which enable us to compute the Jacobian matrix $J(\mathbf{\tilde x};\mathbf{x})=$
	$$
	\left(
	\begin{array}{c|c|c|c}
	I_{n} & \mathrm{O} & \mathrm{O}& \mathrm{O} \\
	\hdotsfor[.5]{4}\\
	\left(\frac{\partial\tilde x_{n-\ell+k}}{\partial x_j}\right)_{\substack{\ell+1\leqslant k\leqslant m \\ 1\leqslant j\leqslant n}} &
	\left(\frac{\partial\tilde x_{n-\ell+k}}{\partial x_j}\right)_{\substack{\ell+1\leqslant k\leqslant m \\ n+1\leqslant j\leqslant n-\ell+m}} & \left(\frac{\partial\tilde x_{n-\ell+k}}{\partial x_j}\right)_{\substack{\ell+1\leqslant k\leqslant m \\ n-\ell+m+1\leqslant j\leqslant n-r+m}} & \mathrm{O}  \\
	\hdotsfor[.5]{4}\\
	\mathrm{O} & \mathrm{O} & I_{\ell-r}  & \mathrm{O}\\
	\hdotsfor[.5]{4}\\
	\mathrm{O} & \mathrm{O}  & \mathrm{O} & I_{d-(n+m-r)} \\
	\end{array}
	\right).
	$$
	An easy computation shows that $\left(\frac{\partial\tilde x_{n-\ell+k}}{\partial x_j}\right)_{\substack{\ell+1\leqslant k\leqslant m \\ n+1\leqslant j\leqslant n-\ell+m}}=(\gamma_{ik})_{\substack{\ell+1\leqslant i\leqslant m \\\ell+1\leqslant k\leqslant m}}$, and hence after performing elementary row operations we obtain
	$$J(\mathbf{\tilde x};\mathbf{x})\leftrightarrow\left(
	\begin{array}{c|c|c}
	I_{n} & \mathrm{O} & \mathrm{O} \\
	\hdotsfor[.5]{3}\\
	\mathrm{O} &(\gamma_{ik})_{\substack{\ell+1\leqslant i\leqslant m \\\ell+1\leqslant k\leqslant m}} & \mathrm{O} \\
	\hdotsfor[.5]{3}\\
	\mathrm{O} & \mathrm{O} & I_{d-(n+m-\ell)} \\
	\end{array}
	\right).
	$$
	But $\mathrm{rank~}(\gamma_{ik})_{\substack{\ell+1\leqslant i\leqslant m \\\ell+1\leqslant k\leqslant m}}=m-\ell$ since
	\begin{equation*}
	\det J(\mathbf{\tilde y};\mathbf{y}) = \det \left(
	\begin{array}{c|c}
	I_{\ell} & \mathrm{O}_{\ell\times (m-\ell)} \\
	\hdotsfor[.5]{2}\\
	(\gamma_{ik})_{\substack{\ell+1\leqslant i\leqslant m \\1\leqslant k\leqslant \ell}} & (\gamma_{ik})_{\substack{\ell+1\leqslant i\leqslant m \\\ell+1\leqslant k\leqslant m}}
	\end{array}
	\right)  = \det \left(
	\begin{array}{c|c}
	I_{\ell} & \mathrm{O}_{\ell\times (m-\ell)} \\
	\hdotsfor[.5]{2}\\
	\mathrm{O}_{m-\ell\times \ell} & (\gamma_{ik})_{\substack{\ell+1\leqslant i\leqslant m \\\ell+1\leqslant k\leqslant m}}
	\end{array}
	\right)\neq 0.
	\end{equation*}
	Thus $J(\mathbf{\tilde x};\mathbf{x})$ is a full rank matrix and hence $\mathbf{\tilde x}$ is a permissible system of parameters at $p$. % on $D_X$.
\end{proof}

%%%%%%%%%%%%%%%%%%%%%%%%%%%%%%%%%%%%%%%%%%%%%%%%%%%%%%%%%%

\subsection{Principalization}
Let $\varphi:(X,U_{\a},D_{\a})_{\a\in I}\to (Y,V_{\a},E_{\a})_{\a\in I}$ be a locally toroidal morphism, and suppose $\pi:Y_1\to Y$ is a permissible blowup in the ERS of  
$\widetilde{\mathcal{E}}_0=\sum_{\alpha \in I}\overline{E}_{\alpha}$ with center $Z\subset Y$. The key observation is that indeterminacy of the rational map $\pi^{-1}\circ\varphi:X \dashrightarrow Y_1$ coincides with the locus of points where $\mathcal I_Z\mathcal O_X$ is not locally principal, i.e., 
$$
\mathbf{W_{Z}(X):=\{p\in X\ |\ \mathcal I_{Z}\mathcal{O}_{X,p} \mbox{\ \ \textbf{is not principal}}\}}.
$$
% indeterminate 
By \cite[Lemma 13.8]{C4}, $\mathcal I_Z\mathcal O_X$ can be factored as $\mathcal I_Z\mathcal O_X = \mathcal O_X(-F) \mathcal N$, where $F$ is an effective divisor on $X$\footnote{Here $\mathcal O_X(-F)$ is the sheaf associated to $-F$ which is invertible since $X$ is nonsingular (see \cite[Section?]{C4}).}, and $\mathcal N$ is an ideal sheaf on $X$ such that $\mathrm{Supp}\,(\mathcal O_X/\mathcal N)=W_{Z}(X)$\footnote{In the proof of Proposition \ref{NP Locus is SNC}, we will observe that $\mathcal I_Z\mathcal O_X$ is a locally monomial ideal sheaf. This implies that $\mathcal N$ is locally monomial as well, and $F$ is a SNC divisor}. We apply the principalization algorithm of ideal sheaves given in \cite[Theorem 6.35]{C2} to $\mathcal N$ in order to resolve the indeterminacy of $\pi^{-1}\circ\varphi$. The finite sequence of transformation constructed by this algorithm is called a \textbf{\textit{strong principalization}} of (the ideal sheaf $\mathcal N$ of) $W_Z(X)$.
% We review the properties of this algorithm we will need in the following Definition/Remark.
%\begin{Definition}\label{Strong Principalization}
%	Let $X$ be a nonsingular $\mathcal K$--variety of dimension $d$, and suppose that $\mathcal N\subset \mathcal O_X$ is an ideal sheaf on $X$. A  is the finite sequence 
%	\begin{equation}
%	\end{equation}
%	of transformation constructed in \cite[Theorem 6.35]{C2}
%
%\end{Definition}

% \begin{Remark}
% basic properties of a permissible center
% We note that \cite[Theorem 6.35]{C2} guarantees
% mention Goward principalization
%\end{Remark}

\begin{Proposition}\label{NP Locus is SNC}
	Let $\varphi:(U,D) \to (V,E)$ be a toroidal morphism. Suppose that $Z\subset V$ is a $c$ codimensional $(\ell,\bar\ell)$--SV for $E$ at $q\in V$. The nonprincipal locus $W_Z(U)$ is a SNC subscheme of $U$ whose irreducible components, say $T$, make SNCs with $D$ and $2 \leqslant \mathrm{codim}_U T \leqslant \dim V$.
\end{Proposition}
%%%%%%%%%%%%%%%%%%%%%%%%%%%%
\begin{proof}
	We will first observe that $\mathcal I_Z\mathcal O_U$ is a locally monomial ideal sheaf, i.e., for all $p\in U$, $\mathcal I_Z\mathcal O_{U,p}$ is generated by monomials in an appropriate regular system of parameters in $\mathcal O_{U,p}$.
	
	We note that if $\ell=0$ %which implies $\bar\ell=0$
	, i.e., $Z$ is a $(0,0)$--SV, then $\varphi$ is smooth at $p\in\varphi^{-1}(Z)$ since it is toroidal. So there exists a regular system of parameters $\mathbf{y}=(y_1,\dots,y_m)$ in $\mathcal O_{V,q}$ such that $y_1=\cdots=y_c=0$ are local equations of $Z$ at $q$, and $\mathbf{y}$ forms a part of a regular system of parameters, say $\mathbf{x}=(x_1,\dots,x_d)$, in $\mathcal O_{U,p}$. Thus, after possibly permuting $x_1,\dots,x_d$, we have
	\begin{equation*}\tag{$\mathfrak{np}\,(0,0)$}\label{0,0}
	\mathcal I_Z\mathcal O_{U,p}=\langle x_1,\dots, x_c \rangle .
	\end{equation*}
	Now suppose that $\ell\geqslant 1$, i.e., $q\in Z\cap E$, and so $\bar\ell\geqslant 0$. By Lemma \ref{tf in terms of perpars for SVs}, there exist permissible parameters $y_1,\dots,y_m$ for $Z$ at $q$ on $E$ and permissible parameters $x_1,\dots,x_d$ at $p\in\varphi^{-1}(q)$, and there exist an exponent matrix $(a_{ij})_{\substack{1\leqslant i\leqslant \ell \\ 1\leqslant j\leqslant n}}\in\mathbb N^{\ell\times n}$ satisfying (\ref{ex matrix condition}), and units $\delta_1,\dots,\delta_{\ell}$ in $\hat{\mathcal O}_{U,p}$ such that
	\begin{equation*}
	y_i=\left\{
	\begin{array}{ll}
	\delta_i x_1^{a_{i1}} \cdots  x_n^{a_{in}}, & 1\leqslant i\leqslant \ell \\
	\ \ \ x_{n-\ell+i}, & \ell+1\leqslant i\leqslant m.
	\end{array}
	\right.
	\end{equation*}
	Recall from the Definition \ref{SV PerPar} that $y_1\cdots y_{\ell}=0$ is a local equation of $E$ at $q$, and local equations of $Z$ at $q$ are
	$$
	y_1=\cdots=y_{\bar \ell}= y_{\ell+1}=\cdots= y_{\ell+(c-\bar\l)}=0.
	$$
	Thus
	\begin{equation}\label{IZ2}%\tag{$\mathfrak{np}\,(0,0)$}
	\mathcal I_Z\hat{\mathcal O}_{U,p}=\langle \delta_1 x_1^{a_{11}} \cdots  x_n^{a_{1n}}, \dots,\delta_{\bar\ell} x_1^{a_{\bar\ell 1}} \cdots  x_n^{a_{\bar\ell n}}, x_{n+1},\dots,x_{n+(c-\bar\ell)} \rangle.
	\end{equation}
	We also note that $D$ is a SNC divisor\footnote{We may assume that $D$ is reduced.} on $U$ and there exist a regular system of parameters $\bar{\mathbf x}=(\bar x_1,\dots,\bar x_d)$ in $\mathcal O_{U,p}$ and $e\in\mathbb N$ such that $\bar x_1\cdots\bar x_e=0$ is a local equation of $D$ at $p$. Recall that $p$ is an $n$--point and $x_1\cdots x_n=0$ is a local equation of $D$ in $\hat{\mathcal O}_{U,p}$. So there exists a unit series $\gamma \in \hat{\mathcal O}_{U,p}$ such that 
	$$
	\bar x_1\cdots\bar x_e=\gamma\, x_1\cdots x_n.
	$$
	Since $x_j$'s and $\bar x_j$'s are irreducible in $\hat{\mathcal O}_{U,p}$, we have $e=n$, and there exist unit series $\gamma_j \in \hat{\mathcal O}_{U,p}$ such that, after possibly reindexing $\bar x_j$'s, for all $1\leqslant j\leqslant n$ we have $x_j=\gamma_j\bar x_j$. Thus, for $1\leqslant i\leqslant \ell$, 
	$$
	y_i=\delta_i \gamma_1^{a_{i1}}\cdots\gamma_n^{a_{in}}\bar x_1^{a_{i1}}\cdots\bar x_n^{a_{in}},
	$$
	where $\bar\delta_i:=\delta_i \gamma_1^{a_{i1}}\cdots\gamma_n^{a_{in}}$ is a unit in $\hat{\mathcal O}_{U,p}$. In fact, $\bar\delta_i$'s are units in $\mathcal O_{U,p}$ since, for $1\leqslant i\leqslant \ell$,
	\begin{equation}
	\bar\delta_i=\frac{y_i}{\bar x_1^{a_{i1}}\cdots\bar x_n^{a_{in}}}\in QF(\mathcal O_{U,p})\cap\hat{\mathcal O}_{U,p}=\mathcal O_{U,p}\mbox{ by \cite [Lemma 2.1]{C1}}.
	\end{equation}
	So, for $1\leqslant i\leqslant \ell$, $\bar y_i:={\bar\delta_i}^{-1}y_i \in \mathcal O_{V,q}$ and we have 
	$$
	\bar y_i = \bar x_1^{a_{i1}}\cdots\bar x_n^{a_{in}}.
	$$
	In addition, after possibly permuting $\bar x_{n+1},\dots,\bar x_d$,
	$$
	\bar{\bar{\mathbf x}}:=(\bar x_1,\dots,\bar x_{n},x_{n+1},\dots,x_{n+(m-\l)},\bar x_{n+m-\l+1},\dots,\bar x_d)
	$$
	is a regular system of parameters in $\hat{\mathcal O}_{U,p}$, and so the ideal
	$$
	\langle \bar x_1,\dots,\bar x_{n},x_{n+1},\dots,x_{n+(m-\l)},\bar x_{n+m-\l+1},\dots,\bar x_d \rangle\hat{\mathcal O}_{U,p}
	$$
	is the maximal ideal of $\hat{\mathcal O}_{U,p}$. We note that for $\ell+1\leqslant i\leqslant m$, $y_i=x_{n-\l+i}\in\mathcal O_{U,p}$, and by faithful flatness of $\hat{\mathcal O}_{U,p}$ over $\mathcal O_{U,p}$
	$$
	\langle \bar x_1,\dots,\bar x_{n},x_{n+1},\dots,x_{n+(m-\l)},\bar x_{n+m-\l+1},\dots,\bar x_d \rangle \mathcal O_{U,p}
	$$
	is also the maximal ideal of $\mathcal O_{U,p}$. Hence $\bar{\bar{\mathbf x}}$ is a regular system of parameters in $\mathcal O_{U,p}$. Rewriting (\ref{IZ2}) in terms of $\bar{\bar{\mathbf x}}$ and $\bar{\mathbf y}:=(\bar y_1,\dots,\bar y_{\ell},y_{\ell+1},\dots,y_m)$, we obtain
	\begin{equation}\label{IZ2'}%\tag{$\mathfrak{np}\,(0,0)$}
	\mathcal I_Z \mathcal O_{X,p}=\langle \bar x_1^{a_{11}} \cdots  \bar x_n^{a_{1n}}, \dots,\bar x_1^{a_{\bar\ell 1}} \cdots  \bar x_n^{a_{\bar\ell n}}, x_{n+1},\dots,x_{n+(c-\bar\l)} \rangle.
	\end{equation}
	Writing $M:=\gcd(\bar x_1^{a_{11}} \cdots  \bar x_n^{a_{1n}}, \dots,\bar x_1^{a_{\bar\ell 1}} \cdots  \bar x_n^{a_{\bar\ell n}}, x_{n+1},\dots,x_{n+(c-\bar\l)})$, we have that 
	$$
	\mathcal N \mathcal O_{X,p} =\mathcal I_Z \mathcal O_{X,p} : M.
	$$
	Therefore by \cite[Proposition 1.2.2]{HH}, if $\bar\ell=c$,
	\begin{equation}\label{Np1}
	\mathcal N \mathcal O_{X,p} =
	\langle \bar x_1^{\bar a_{11}} \cdots  \bar x_n^{\bar a_{1n}}, \dots,\bar x_1^{\bar a_{\bar\ell 1}} \cdots \bar x_n^{\bar a_{\bar\ell n}} \rangle,
	\end{equation}
	where $\bar a_{ij}=a_{ij}-\min\{a_{ij}|1\leqslant i \leqslant \bar\ell\}$, for $1\leqslant i \leqslant \bar\ell$ and $1\leqslant j \leqslant n$; and if $\bar\ell<c$,
	\begin{equation}\label{Np2}
	\mathcal N \mathcal O_{X,p} = \langle \bar x_1^{a_{11}} \cdots  \bar x_n^{a_{1n}}, \dots,\bar x_1^{a_{\bar\ell 1}} \cdots  \bar x_n^{a_{\bar\ell n}}, x_{n+1},\dots,x_{n+(c-\bar\ell)} \rangle.
	\end{equation}
	Thus $\mathcal N$ is a locally monomial ideal sheaf in $\mathcal O_U$, i.e., $W_Z(U)$ is a SNC subscheme of $U$.
	
	In addition, using \cite[Theorem 1.3.1]{HH}, the monomial ideals in 
	(\ref{Np1}) and (\ref{Np2}) can be written as intersection of monomial ideals generated by pure powers of the variables, i.e., irreducible monomial ideals\footnote{A monomial ideal is called \textbf{\textit{irreducible}} if it cannot be written as proper intersection of two other monomial ideals.}. So the defining ideal of an irreducible component $T$ of $W_Z(U)$ containing $p$ has one of the following forms:
	\subsubsection*{np(1)} If $\bar\ell=c$ and (\ref{Np1}) holds, there exist $k\in\mathbb N$\footnote{$k\leqslant \min\{n,\bar\ell\}$}, distinct $j_1,\dots,j_k\in [n]$ and $i_1,\dots,i_k\in[\bar\ell]$ such that for all $t\in[k]$, $\bar a_{i_tj_t}=a_{i_tj_t}-m_{j_t}\neq 0$ and
	\begin{equation*}\tag{$\mathfrak{np}\,(1)$}\label{np1}
	\sqrt{\mathcal I_T\mathcal O_{X,p}}=\langle \bar x_{j_1},\dots, \bar x_{j_k} \rangle;
	\end{equation*}
	\subsubsection*{np(2)} if $\bar\ell<c$ and (\ref{Np2}) holds, there exist $k\in\mathbb N$, distinct $j_1,\dots,j_k\in [n]$ and $i_1,\dots,i_k\in[\bar\ell]$ such that for all $t\in[k]$, $ a_{i_tj_t}\neq 0$ and
	\begin{equation*}\tag{$\mathfrak{np}\,(2)$}\label{np2}
	\sqrt{\mathcal I_T\mathcal O_{X,p}}=\langle \bar x_{j_1},\dots, \bar x_{j_k}, x_{n+1},\dots,x_{n+(c-\bar\ell)} \rangle.
	\end{equation*}
	This shows that irreducible components of $W_Z(U)$, say $T$, make SNCs with $D$, and clearly, $$
	2 \leqslant \mathrm{codim}_U T \leqslant \mathrm{codim}_V Z\leqslant \dim V.
	$$
\end{proof}

\subsection{Quasi--toroidal Forms}
We are now ready to study the effect of a principalization sequence on toroidal forms in the following lemmas.
\begin{Lemma}\label{QTF}
	Suppose that $\varphi:(U,D)\to (V,E)$ is a toroidal morphism. Let $q\in E$ be an $\ell$--point, and let $Z\subset V$ be a $c$ codimensional $(\ell,\bar\ell)$--SV for $E$ at $q$. Suppose that
	\begin{equation}\label{P}
	U_{n_0}\xrightarrow{\lambda_{n_0}}U_{n_0-1}\cdots\xrightarrow{}U_k\xrightarrow{\lambda_k}U_{k-1}\xrightarrow{}\cdots\to U_1\xrightarrow{\lambda_1}U_0=U
	\end{equation}
	is a strong principalization sequence of %$\mathcal I_Z\mathcal O_U$. 
	$W_Z(U)$. We write $\Lambda_k=\lambda_1\circ\cdots\circ\lambda_k$, for $k \geqslant 1$, and $\Lambda_0=id_U$. Then:
	\begin{enumerate}
		\item\label{QTF part21} For all $0 \leqslant k\leqslant n_0$, $D_k=\Lambda_k^*(D)_{\mathrm{red}}$ is a SNC divisor on $U_k$.\\
		
		\item\label{QTF part2} For all $0 \leqslant k\leqslant n_0$, if $p\in (\varphi\circ\Lambda_k)^{-1}(q)$ is an $n$-point for $D_k$, there exist permissible parameters $x_1,\dots,x_d$ at $p$ on $D_k$ %\footnote{We note that permissible parameters $x_1,\dots,x_d$ are in $\hat{\mathcal O}_{U_k,p}$.}
		, permissible parameters $y_1,\dots,y_m$ for $Z$ at $q$ on $E$, and there exist a matrix $\mathbf{a}=(a_{ij})_{\substack{1\leqslant i\leqslant \ell+s \\ 1\leqslant j\leqslant n}}\in\mathbb N^{(\ell+s)\times n}$, and unit series $\delta_1,\dots,\delta_{\ell+s}\in\hat{\mathcal O}_{U_k,p}$, where $s:=c-\bar\ell\geqslant 0$, and there exist $\beta_{n+1},\dots,\beta_{n+s}\in\mathcal K$ such that one of the following forms holds:
		\begin{equation}\label{QTF1}
		y_i=\left\{
		\begin{array}{ll}
		\delta_i x_1^{a_{i1}} \cdots  x_n^{a_{in}}, & 1\leqslant i\leqslant \ell \\
		\delta_i x_1^{a_{i1}} \cdots  x_n^{a_{in}}(x_{n-\ell+i}+\beta_{n-\ell+i}), & \ell+1\leqslant i\leqslant \ell+s \\
		\ \ \ x_{n-\ell+i}, & \ell+s+1\leqslant i\leqslant m
		\end{array}
		\right.
		\end{equation}
		where for all $j\in [n]$, $\sum_{i=1}^{\ell}a_{ij}\neq 0$, and for all $i\in [\ell]$, $\sum_{j=1}^{n}a_{ij}\neq 0$; or
		\begin{equation}\label{QTF2}
		y_i=\left\{
		\begin{array}{ll}
		\delta_i x_1^{a_{i1}} \cdots  x_n^{a_{in}}, & 1\leqslant i\leqslant \ell+1 \\
		\delta_i x_1^{a_{i1}} \cdots  x_n^{a_{in}}(x_{n-\ell+i}+\beta_{n-\ell+i}), & \ell+2\leqslant i\leqslant \ell+s \\
		\ \ \ x_{n-\ell+i}, & \ell+s+1\leqslant i\leqslant m
		\end{array}
		\right.
		\end{equation}
		where for all $j\in[n]$, $\sum_{i=1}^{\ell}a_{ij}\neq 0$, and for all $i\in[\ell+1]$, $\sum_{j=1}^{n}a_{ij}\neq 0$. Furthermore, in both cases, if $s>0$, 
		\begin{equation}\label{QTF condition}
		\mbox{for all } j\in [n],\ \ a_{(\ell+1)j}=\cdots=a_{(\ell+s)j}=\min\{a_{ij}\}_{i\in I},
		\end{equation}
		where $I:= [\bar\ell]\cup \{\ell+1,\dots,\ell+s\}$.\\
		
		\item\label{QTF part3} For all $0 \leqslant k < n_0$, the nonprincipal locus $W_Z(U_k)$ is a SNC subscheme of $U_k$ whose irreducible components, say $W_k\subset W_Z(U_k)$, make SNCs with $D_k$. More precisely, if $p\in W_Z(U_k)$, for some $k\geqslant 0$, then (\ref{QTF1}) with $\beta_{n+1}=\cdots=\beta_{n+s}=0$ holds at $p$, and if $\bar\ell=c$, 
		\begin{equation}
		\mathcal N_{k,p}:=\mathcal N \hat{\mathcal O}_{U_k,p} = \langle  x_1^{\bar a_{11}} \cdots  x_n^{\bar a_{1n}}, \dots, x_1^{\bar a_{\bar\ell 1}} \cdots  x_n^{\bar a_{\bar\ell n}} \rangle
		\end{equation}
		where $\bar a_{ij}=a_{ij}-\min\{a_{ij}|i\in [\bar\ell]\}$, for $1\leqslant i \leqslant \bar\ell$ and $1\leqslant j \leqslant n$; and if $\bar\ell<c$,
		\begin{equation}
		\mathcal N_{k,p} = \langle x_1^{a_{11}} \cdots  x_n^{a_{1n}}, \dots, x_1^{a_{\bar\ell 1}} \cdots  x_n^{a_{\bar\ell n}}, x_{n+1},\dots,x_{n+s} \rangle.
		\end{equation}
	\end{enumerate}
\end{Lemma}

\begin{proof}
	We will prove the lemma by induction on the length $k$ of the principalization sequence of $\mathcal I_Z\mathcal O_U$. When $k=0$, Theorem \ref{TF} and Lemma \ref{NP Locus is SNC} tell us that the claims hold on $U$\footnote{We note that $\varphi:(U,D)\to (V,E)$ is toroidal and the toroidal form (\ref{TFeq}) is just (\ref{QTF1}) with $a_{(\ell+t)j}=0$ for $t\in[s]$ and $j\in[n]$, and $\beta_{n+t}=0$ for $t\in [s]$.} and we take this as the base case for our induction. Suppose that $p\in(\varphi\circ\Lambda_k)^{-1}(q)=\lambda_k^{-1}((\varphi\circ\Lambda_{k-1})^{-1}(q))$ for some $k\geqslant 0$, and the claims of the theorem hold on $U_{k-1}$. Let $T_{k-1}\subset W_Z(U_{k-1})$ be the center of blowing up $\lambda_{k}:U_k\to U_{k-1}$\footnote{We note that the center $T_{k-1}$ of blowing up $\lambda_{k}:U_k\to U_{k-1}$ may be a proper irreducible nonsingular subvariety of an irreducible component $W_{k-1}$ of $W_Z(U_{k-1})$}. Since $T_{k-1}$ is nonsingular and makes SNCs with $D_{k-1}$, the divisor $D_k=\lambda_k^{-1}(D_{k-1})$ is a SNC divisor on $U_k$.
	
	To complete the proof, it suffices to show that (\ref{QTF part2}) holds for $p$, then the argument of Lemma \ref{NP Locus is SNC} can be adopted easily to show (\ref{QTF part3}). Since $\lambda_k$ is an isomorphism out of its center, we may assume that $p\in\lambda_k^{-1}(T_{k-1})$, i.e., $\bar p=\lambda_k(p)\in T_{k-1}$. Suppose that $\bar p$ is an $n$--point for $D_{k-1}$. So by the induction hypothesis, (\ref{QTF part3}) holds at $\bar p$ and there exist unit series $\bar\delta_i\in\hat{\mathcal O}_{U_{k-1},\bar p}$, permissible parameters $y_1,\dots,y_m$ for $Z$ at $q$ on $E$, permissible parameters $\bar x_1,\dots,\bar x_d$ at $\bar p$ for $D_{k-1}$, and an exponent matrix $\mathbf{a}=(a_{ij})_{\substack{1\leqslant i\leqslant \ell+s \\ 1\leqslant j\leqslant n}}\in\mathbb N^{(\ell+s)\times n}$ satisfying (\ref{ex matrix condition}) such that
	\begin{equation}\label{np proof}
	y_i=\left\{
	\begin{array}{ll}
	\bar\delta_i \bar x_1^{a_{i1}} \cdots  \bar x_n^{a_{in}}, & 1\leqslant i\leqslant \ell \\
	\bar\delta_i \bar x_1^{a_{i1}} \cdots  \bar x_n^{a_{in}}\bar x_{n-\ell+i}, & \ell+1\leqslant i\leqslant \ell+s \\
	\ \ \ \bar x_{n-\ell+i}, & \ell+s+1\leqslant i\leqslant m
	\end{array}
	\right.
	\end{equation}
	where $s=c-\bar\ell\geqslant 0$, and recall that 
	$$
	y_1=\cdots=y_{\bar\ell}=y_{\ell+1}=\cdots=y_{\ell+s}=0
	$$
	are local equations of $Z$ at $q$. Without loss of generality, we may assume that local equations of $T_{k-1}$, the center of blowing up $\lambda_k:U_k\to U_{k-1}$, at $\bar p$ are 
	$$
	\bar x_1=\cdots=\bar x_{e-s}=\bar x_{n+1}=\cdots=\bar x_{n+s}=0
	$$
	where $e$ denotes the codimension of $T_{k-1}$ in $U_{k-1}$. By \cite[Theorem 10.19]{C4}, there exist regular parameters $x_1,\dots,x_d\in\hat{\mathcal O}_{U_k,p}$ at $p\in \lambda_k ^{-1}(\bar p)$, and there exist $j_0\in J:=[e-s]\cup \{n+1,\dots,n+s\}$ and $\beta_j\in\mathcal K$, for each $j\in J$, such that
	\begin{equation}\label{blupeqs}
	\bar x_j=\left\{
	\begin{array}{ll}
	x_{j_0}( x_j+\beta_j) & \hbox{for $ j\in J$ and $j\neq j_0$,} \\
	x_{j_0} & \hbox{for $j=j_0$,} \\
	x_j & \hbox{for $j\in [d]\setminus J$,}
	\end{array}
	\right.
	\end{equation}
	and $x_{j_0}=0$ is a local equation of the exceptional divisor of $\lambda_k$ at $p$. To obtain all the possible forms of $\varphi\circ\Lambda_k(p)$, it suffices to  consider $j_0=1$ and $j_0=n+1$.\\
	
	\textbf{Case 1: $j_0=1$.} First suppose there exist regular parameters $x_1,\dots,x_d\in\hat{\mathcal O}_{U_k,p}$ and $\beta_j\in\mathcal K$, $j\in J$, such that (\ref{blupeqs}) holds at $p$ and $x_1=0$ is a local equation of the exceptional divisor of $\lambda_k$. Since $\bar x_1 \cdots \bar x_n=0$ is a local equation of $D_{k-1}$ at $\bar p$, so
	$$x_1\times \prod_{2\leqslant j\leqslant e-s} x_1(x_j+\beta_j)\times \prod_{e-s+1\leqslant j\leqslant n} x_j=0$$
	is a local equation of $\lambda_k^*(D_{k-1})$. Let 
	\begin{align}
	\label{J1} J_1 & :=\{j|2\leqslant j\leqslant e-s\mbox{ and }\beta_j=0\},\\
	\label{J2} J_2 & :=\{j|2\leqslant j\leqslant e-s\mbox{ and }\beta_j\in\mathcal K^{\times}\}
	\end{align}
	and suppose that $j_1:=|J_1|+1$, so that $1\leqslant j_1\leqslant e-s$ since $0\leqslant |J_1|\leqslant e-s-1$. After possibly permuting $x_2,\dots,x_{e-s}$, we may assume that $\beta_2=\cdots=\beta_{j_1}=0$ and $\beta_{j_1+1},\dots,\beta_{e-s}\in\mathcal K^{\times}$. Hence 
	$$x_1 x_2 \cdots x_{j_1} x_{e-s+1}\cdots x_n=0$$
	is a local equation of $D_k=\lambda_k^*(D_{k-1})_{\rm red}$, and writing $N:=n-(e-s-j_1)$, $p$ is an $N$--point for $D_k$. We note that $n-e+s+1 \leqslant N\leqslant n$. Now by substituting (\ref{blupeqs}) in (\ref{np proof}), for $1\leqslant i\leqslant \ell$, we obtain
	$$ 
	y_i=\delta_i \times x_1^{\sum_{j=1}^{e-s} a_{ij}} \times \prod_{j\in J_1} x_j^{a_{ij}} \times \prod_{j=e-s+1}^n  x_j^{a_{ij}} \times \prod_{j\in J_2}(x_j+\beta_j)^{a_{ij}}, 
	$$
	for $\ell+1\leqslant i\leqslant \ell+s$, we have
	$$
	y_i=\delta_i \times x_1^{1+\sum_{j=1}^{e-s} a_{ij}} \times \prod_{j\in J_1} x_j^{a_{ij}}  \times \prod_{j=e-s+1}^{n}x_j^{a_{ij}} \times  \prod_{j\in J_2}(x_j+\beta_j)^{a_{ij}} \times (x_{n-\ell+i}+\beta_{n-\ell+i})
	$$
	where $\beta_{n-\ell+i}\in\mathcal K$, and finally, $y_i=x_{n-\ell+i}$ for $\ell+s+1\leqslant i\leqslant m$. Here we abuse notation and we use the same symbol for unit series $\delta_i$'s, in which the variables $\bar x_1,\dots,\bar x_{n+(m-\l)}$ do not appear, before and after the substitution.
	
	We then reindex $x_j$'s and $a_{ij}$'s as follows. Consider the map $\iota:[d]\to [d]$ defined by 
	%\footnote{Recall that we permute $x_2,\dots,x_{e-s}$ such that $J_1=\{2,\dots,j_1\}$ and $J_2=\{j_1+1,\dots,e-s\}$.}:
	\begin{equation}
	\iota(j)=\begin{cases}
	j & \hbox{for } 1\leqslant j\leqslant j_1, \\
	j+(e-s-j_1) & \hbox{for } j_1+1\leqslant j\leqslant d-(e-s-j_1), \\
	j-(d-e+s) & \hbox{for } d-(e-s-j_1)+1\leqslant j\leqslant d, 
	\end{cases}
	\end{equation}
	and let ${\rm x}_j:=x_{\iota(j)}$ for $j\in [d]$, $\beta_j':= \beta_{\iota(j)}$ for $j\in \{N+1,\dots,N+s\}\cup \{d-(e-s-j_1)+1,\dots,d\}$, and
	\begin{equation}
	b_{ij} :=\begin{cases}
	\sum_{k=1}^{e-s} a_{ik} & \hbox{for } 1\leqslant i\leqslant \ell\ ,\ j=1 , \\
	\sum_{k=1}^{e-s} a_{ik} + 1 & \hbox{for } \ell< i\leqslant \ell+s\ ,\ j=1 , \\
	a_{i,\iota(j)} & \hbox{for } 1\leqslant i\leqslant \ell+s\ ,\ j\in \{2,\dots,N\}\cup \{d-(e-s-j_1)+1,\dots,d\}.
	\end{cases}
	\end{equation}
	Thus we obtain
	\begin{equation}
	y_i=\begin{cases}
	\delta_i' {\rm x}_1^{b_{i1}} \cdots  {\rm x}_N^{b_{iN}}, & 1\leqslant i\leqslant \ell \\
	\delta_i' {\rm x}_1^{b_{i1}} \cdots  {\rm x}_N^{b_{iN}}({\rm x}_{N-\ell+i} + \beta_{N-\ell+i}'), & \ell+1\leqslant i\leqslant \ell+s \\
	\ \ \ {\rm x}_{N-\ell+i}, & \ell+s+1\leqslant i\leqslant m	
	\end{cases}
	\end{equation}
	where $\delta_i'$, $i\in [\ell+s]$ are defined by
	\begin{equation}
	\delta_i':= \delta_i \times \prod_{j=d-(e-s-j_1)+1}^d  ({\rm x}_j+\beta'_j)^{b_{ij}}, 
	\end{equation}
	in which the variables ${\rm x}_1,\dots,{\rm x}_{N+m-\l}$ do not appear. (We note for $j\in \{N+1,\dots,N+s\}$, $\beta'_j\in\mathcal K$, and for $j\in\{d-(e-s-j_1)+1,\dots,d\}$, $\beta'_j\in\mathcal K^{\times}$).
	
	It is easy to check that for all $j\in [N]$, $\sum_{i\in[\ell]} b_{ij}>0$, and for all $i\in[\ell]$, $\sum_{j\in[N]} b_{ij}>0$. First observe that for $i\in[\ell]$,
	$$
	\sum_{j=1}^{N} b_{ij} = b_{i1}+\sum_{j=2}^N a_{i,\iota(j)} = \sum_{k=1}^{e-s} a_{ik} + \sum_{k=2}^{j_1} a_{ik} + \sum_{k=e-s+1}^{n} a_{ik}=\sum_{k=1}^{n} a_{ik} + \sum_{k=2}^{j_1} a_{ik}
	$$
	and by the induction hypothesis $\sum_{k=1}^{n} a_{ik}> 0$. So for all $i\in[\ell]$, $\sum_{j\in[N]} b_{ij} >0$. We also have $\sum_{i\in[\ell]} b_{ij} >0$ for all $j$, $2\leqslant j\leqslant N$, by the induction hypothesis. For $j=1$, we note that by Remark \ref{permissible center}, for each $i\in[\bar\ell]$, there exists $k\in[e-s]$ such that $a_{ik}-\min\{a_{ij}|i\in I\}>0$ and hence $a_{ik}>0$. So $b_{i1}=\sum_{k=1}^{e-s} a_{ik}>0$ for all $i\in[\bar\ell]$, and then $\sum_{i=1}^{\ell} b_{i1} > 0 $.
	
	Next, we show that if $s>0$,
	\begin{equation}
	\mbox{for all } j\in [N],\ \ b_{(\ell+1)j}=\cdots=b_{(\ell+s)j}=\min\{b_{ij}\}_{i\in I},
	\end{equation}
	where $I:= [\bar\ell]\cup \{\ell+1,\dots,\ell+s\}$. The induction hypothesis tells us that for all $j$, $1\leqslant j\leqslant n$, 
	$$
	{\rm a}_j :=\min\{a_{ij}\}_{i\in I} = a_{(\ell+1)j}=\cdots=a_{(\ell+s)j},
	$$
	which implies immediately that for $j$, $2\leqslant j\leqslant N$, $b_{(\ell+1)j}=\cdots=b_{(\ell+s)j}=\min\{b_{ij}\}_{i\in I}$. Now suppose that $j=1$. We note that for $i$, $\ell+1\leqslant i\leqslant\ell+s$,
	$$b_{i1}=1+\sum_{k=1}^{e-s} a_{ik}=1+\sum_{k=1}^{e-s} {\rm a}_{k}.$$
	Hence we only need to show that ${\rm b}_1:= 1+\sum_{k=1}^{e-s} {\rm a}_{k}=\min\{b_{i1}\}_{i\in I}$. By Remark \ref{permissible center}, for each $i\in[\bar\ell]$, there exists $k_i\in[e-s]$ such that $a_{ik_i}-\min\{a_{ij}|i\in I\}=a_{ik_i}-{\rm a}_{k_i}>0$, and so we can write
	$$
	{\rm a}_{k_i}<a_{ik_i} \leqslant a_{ik_i}+\sum_{k\in[e-s]\setminus \{k_i\}} (a_{ik}-{\rm a}_k).
	$$
	Thus for each $i\in[\bar\ell]$, $\sum_{k\in[e-s]}{\rm a}_k < \sum_{k\in[e-s]} a_{ik}$,
	which implies that 
	$$
	\sum_{k\in[e-s]}{\rm a}_k < \min\left\{\sum_{k=1}^{e-s} a_{ik}\ |\ i \in [\bar\ell] \right\}.
	$$
	So we have that ${\rm b}_1= 1+\sum_{k=1}^{e-s} {\rm a}_{k}\leqslant\min\{b_{i1}|i\in [\bar\ell]\}$, which means that
	$$
	{\rm b}_1= \min(\{b_{i1}|i\in [\bar\ell]\}\cup\{b_{\ell+1,1},\cdots,b_{\ell+s,1}\}) = \min\{b_{i1}|i\in I\}.
	$$
	Therefore the matrix $(b_{ij})_{\substack{1\leqslant i\leqslant \ell+s \\ 1\leqslant j\leqslant N}}$ satisfies condition (\ref{QTF condition}) and quasi-toroidal form (\ref{QTF1}) holds for $p$.\\ % on $D_k$.\\
	
	\textbf{Case 2: $j_0=n+1$.} Now we suppose that there exist regular parameters $x_1,\dots,x_d\in\hat{\mathcal O}_{U_k,p}$ and $\beta_j\in\mathcal K$, $j\in J$, such that (\ref{blupeqs}) holds at $p$ and $x_{n+1}=0$ is a local equation of the exceptional divisor of $\lambda_k$. Since $\bar x_1 \cdots \bar x_n=0$ is a local equation of $D_{k-1}$ at $\bar p$, so
	$$
	\prod_{1\leqslant j\leqslant e-s} x_{n+1}(x_j+\beta_j)\times \prod_{e-s+1\leqslant j\leqslant n} x_j=0
	$$
	is a local equation of $\lambda_k^*(D_{k-1})$. Let 
	\begin{align}
	J'_1 & :=\{j|1\leqslant j\leqslant e-s\mbox{ and }\beta_j=0\},\\
	J'_2 & :=\{j|1\leqslant j\leqslant e-s\mbox{ and }\beta_j\in\mathcal K^{\times}\}
	\end{align}
	and suppose that $j'_1:=|J'_1|$, $0\leqslant j'_1\leqslant e-s$. After possibly permuting $x_1,\dots,x_{e-s}$, we may assume that $\beta_1=\cdots=\beta_{j'_1}=0$ and $\beta_{j'_1+1},\dots,\beta_{e-s}\in\mathcal K^{\times}$. Hence 
	$$x_1 x_2 \cdots x_{j'_1} x_{e-s+1}\cdots x_n x_{n+1}=0$$
	is a local equation of $D_k=\lambda_k^*(D_{k-1})_{\rm red}$, and writing $N':=n-(e-s-j'_1)+1$, $p$ is an $N'$--point for $D_k$. We note that $n-e+s+1 \leqslant N'\leqslant n+1$. After substituting (\ref{blupeqs}) in (\ref{np proof}), for $1\leqslant i\leqslant \ell$, we obtain
	$$ 
	y_i=\delta_i \times x_{n+1}^{\sum_{j=1}^{e-s} a_{ij}} \times \prod_{j\in J'_1} x_j^{a_{ij}} \times \prod_{j=e-s+1}^n  x_j^{a_{ij}} \times \prod_{j\in J'_2}(x_j+\beta_j)^{a_{ij}};
	$$
	for $i=\ell+1$, we get
	$$
	y_i=\delta_i \times x_{n+1}^{1+\sum_{j=1}^{e-s} a_{ij}} \times \prod_{j\in J'_1} x_j^{a_{ij}}  \times \prod_{j=e-s+1}^{n}x_j^{a_{ij}} \times  \prod_{j\in J'_2}(x_j+\beta_j)^{a_{ij}};
	$$
	for $\ell+1 < i\leqslant \ell+s$, we have
	$$
	y_i=\delta_i \times x_{n+1}^{1+\sum_{j=1}^{e-s} a_{ij}} \times \prod_{j\in J'_1} x_j^{a_{ij}}  \times \prod_{j=e-s+1}^{n}x_j^{a_{ij}} \times  \prod_{j\in J'_2}(x_j+\beta_j)^{a_{ij}} \times (x_{n-\ell+i}+\beta_{n-\ell+i})
	$$
	where $\beta_{n-\ell+i}\in\mathcal K$, and finally, $y_i=x_{n-\ell+i}$ for $\ell+s+1\leqslant i\leqslant m$. We reindex $x_j$'s and $a_{ij}$'s as follows. Consider the map $\iota':[d]\to [d]$ defined by 
	%\footnote{Recall that we permute $x_1,\dots,x_{e-s}$ such that $J'_1=\{1,\dots,j_1\}$ and $J'_2=\{j_1+1,\dots,e-s\}$.}:
	\begin{equation}
	\iota'(j)=\begin{cases}
	j & \hbox{for } 1\leqslant j\leqslant j'_1, \\
	j+(e-s-j'_1) & \hbox{for } j'_1< j\leqslant d-(e-s-j'_1), \\
	j-(d-e+s) & \hbox{for } d-(e-s-j'_1)< j\leqslant d, 
	\end{cases}
	\end{equation}
	and let ${\rm x}_j:=x_{\iota'(j)}$ for $j\in [d]$, $\beta_j'':= \beta_{\iota'(j)}$ for $j$, $N'< j < N'+s$ or $d-(e-s-j' _1)<j\leqslant d$, and
	\begin{equation}
	b'_{ij} :=\begin{cases}
	a_{i,\iota'(j)} & \hbox{for } 1\leqslant i\leqslant \ell+s\ ,\ 1\leqslant j < N'\mbox{ or } d-(e-s-j'_1)<j\leqslant d,\\
	\sum_{k=1}^{e-s} a_{ik} & \hbox{for } 1\leqslant i\leqslant \ell\ ,\ j=N' , \\
	\sum_{k=1}^{e-s} a_{ik} + 1 & \hbox{for } \ell< i\leqslant \ell+s\ ,\ j=N'.
	\end{cases}
	\end{equation}
	Therefore, we obtain
	\begin{equation}
	y_i=\begin{cases}
	\delta_i'' {\rm x}_1^{b'_{i1}} \cdots  {\rm x}_{N'}^{b'_{iN'}}, & 1\leqslant i\leqslant \ell+1 \\
	\delta_i'' {\rm x}_1^{b'_{i1}} \cdots  {\rm x}_{N'}^{b'_{iN'}}({\rm x}_{N'-\ell+i} + \beta''_{N'-\ell+i}), & \ell+2\leqslant i\leqslant \ell+s \\
	\ \ \ {\rm x}_{N'-\ell+i}, & \ell+s+1\leqslant i\leqslant m	
	\end{cases}
	\end{equation}
	where $\delta_i''$, $i\in [\ell+s]$ are defined by
	\begin{equation}
	\delta_i'':= \delta_i \times \prod_{j=d-(e-s-j'_1)+1}^d  ({\rm x}_j+\beta''_j)^{b'_{ij}}, 
	\end{equation}
	in which the variables ${\rm x}_1,\dots,{\rm x}_{N'+m-\l}$ do not appear.
	(Note that for $j\in \{N'+2,\dots,N'+s\}$, $\beta''_j\in\mathcal K$, and for $j\in\{d-(e-s-j'_1)+1,\dots,d\}$, $\beta''_j\in\mathcal K^{\times}$). The same argument given in the previous case shows that for all $j\in [N']$, $\sum_{i\in[\ell+1]} b'_{ij}>0$ and for all $i\in[\ell+1]$, $\sum_{j\in[N']} b'_{ij}>0$, and 
	$$
	\mbox{for all } j\in [N'],\ \ b'_{(\ell+1)j}=\cdots=b'_{(\ell+s)j}=\min\{b'_{ij}\}_{i\in I}.
	$$
	Thus quasi-toroidal form (\ref{QTF2}) holds for $p$ on $D_k$.
\end{proof}

\begin{Remark}\label{permissible center}
	With the notation of Lemma \ref{QTF}, let $T_k\subset W_Z(U_k)$ be the center of the permissible blowing up $\lambda_{k+1}:U_{k+1}\to U_k$ in the strong principalization sequence (\ref{P}) of the ideal sheaf $\mathcal N$ of $W_Z(U)$. Suppose that local equations of $T_k$ at $p\in W_Z(U_k)$ are 
	$$
	x_{j_1}=\cdots=x_{j_{e-s}}=x_{n+1}=\cdots=x_{n+s}=0
	$$
	where $e$ denotes the codimension of $T_k$ in $U_k$, and $\{j_1,\dots,j_{e-s}\}\subseteq [n]$\footnote{Recall that we use the convention that $\{n+1,\dots,n+s\}=\emptyset$ if $s=0$, i.e., $c=\bar\ell$, so that in this case $x_{n+1},\dots,x_{n+s}$ are not contained in the local equations of $T_k$).}. Consider the matrix $\mathbf{w}\in\mathbb N^{c\times e}$ defined by 
	$$
	\mathbf{w}=(\omega_{fg})_{\substack{1\leqslant f\leqslant c \\ 1\leqslant g\leqslant e}}:= \left(
	\begin{array}{c|c}
	(a_{ij_t})_{\substack{1\leqslant i\leqslant \bar\ell \\ 1\leqslant t\leqslant e-s}} & \mathrm{O}_{\bar\ell\times s} \\
	\hdotsfor[.5]{2}\\
	(a_{ij_t})_{\substack{\ell+1\leqslant i\leqslant \l+s \\1\leqslant t\leqslant e-s}} & I_{s}
	\end{array}
	\right),   
	$$
	and let $\mathbf{m}=(m_{fg})_{c\times e}:=(\min\{\omega_{tg}|t\in[c]\})_{\substack{1\leqslant f\leqslant c \\ 1\leqslant g\leqslant e}}$. Then $\overline{\mathbf{w}}=(\bar{\omega}_{fg})_{\substack{1\leqslant f\leqslant c \\ 1\leqslant g\leqslant e}}:=\mathbf{w-m}$ is a nonzero row and nonzero column matrix.
\end{Remark} 

\begin{Lemma}\label{QTF0}
	Suppose that $\varphi:(U,D)\to (V,E)$ is a toroidal morphism. 
	%%%%%%%%%%%%%%%%%%%%5
	Let $Z\subset V$ be a $c$ codimensional (0,0)--SV for $E$ at a 0--point $q\in V\setminus E$. Suppose that
	\begin{equation}\label{P0}
	U_{n_0}\xrightarrow{\lambda_{n_0}}U_{n_0-1}\cdots\xrightarrow{}U_k\xrightarrow{\lambda_k}U_{k-1}\xrightarrow{}\cdots\to U_1\xrightarrow{\lambda_1}U_0=U
	\end{equation}
	is a SPS of %$\mathcal I_Z\mathcal O_U$. 
	$W_Z(U)$. We write $\Lambda_k=\lambda_1\circ\cdots\circ\lambda_k$, for $k \geqslant 1$, and $\Lambda_0=id_U$. Then:
	\begin{enumerate}
		\item For all $0 \leqslant k\leqslant n_0$, $p\in (\varphi\circ\Lambda_k)^{-1}(q)$ is also a 0--point for $D_k=\Lambda_k^*(D)_{\mathrm{red}}$, and there exist regular system of parameters $y_1,\dots,y_m$ at $q$, and regular system of parameters $x_1,\dots,x_d$ at $p$ such that one of the following possibilities holds:
		\begin{equation}\label{QTF eq 00}
		y_i=x_i, \qquad \qquad \qquad  \mbox{ for } 1\leqslant i\leqslant m; \mbox{ or}
		\end{equation}
		\begin{equation}\label{QTF eq 01}
		y_i=\begin{cases}
		x_c(x_i+\beta_i) & \mbox{ for } 1\leqslant i\leqslant c-1, \\
		x_i & \mbox{ for } c\leqslant i\leqslant m,
		\end{cases}
		\end{equation}
		where $\beta_1,\dots,\beta_{c-1}\in\mathcal K$.
		
		\item If $p\in W_Z(U_k)$, for some  $0 \leqslant k < n_0$, then (\ref{QTF eq 00}) holds at $p$, and
		\begin{equation}
		\mathcal N_{k,p}:=\mathcal N \hat{\mathcal O}_{U_k,p} = \langle  x_1 , \dots ,  x_c\rangle.
		\end{equation}
	\end{enumerate}
\end{Lemma}

\begin{proof}
	The proof of this lemma is completely to the proof of Lemma \ref{QTF}.
\end{proof}

\section{Main Results}

%In this section, we prove the existence of toroidalization of locally toroidal morphisms in arbitrary dimension. This result is a generalization of \cite[Theorem 4.2]{Ha}, \cite[Theorem 3.19]{A1}, and \cite[Theorem 6.3]{A3}.

\begin{Lemma}\label{tor after principalization}
	Suppose that $\varphi:(U,D)\to (V,E)$ is a toroidal morphism of nonsingular varieties. Let $Z\subset V$ be a nonsingular subvariety of codimension $c$, $c\geqslant 2$, which makes SNCs with $E$. Suppose that $\pi:V_1\to V$ is the blowup of $Z$, and suppose that $\lambda:U_1\to U$ is a SPS of $W_Z(U)$. Then $D_1=\lambda^{-1}(D)$ and $E_1=\pi^{-1}(E)$ are SNC divisors on $U_1$ and $V_1$ respectively, and there exist a toroidal morphism $\phi_1:(U_1,D_1)\to (V_1,E_1)$ such that the diagram
	\begin{equation}\label{LocTor}
	\begin{CD}
	(U_1,D_1)@>\phi_1>> (V_1,E_1)\\
	@V\lambda VV @VV\pi V\\ (U,D) @>>\varphi> (V,E)
	\end{CD}
	\end{equation}
	commutes.
\end{Lemma}

\begin{proof}
	%	$q_1:=\phi_1(p_1)=\pi^{-1}\varphi\lambda(p_1)$,
	Given a point $p_1\in U_1$, we write $p:=\lambda(p_1)$, and $q:=\varphi(p)$. If $q\not\in Z$, then $p\not\in W_Z(U)$ since $W_Z(U)\subseteq\varphi^{-1}(Z)$. In this case, $\pi$ and $\lambda$ are isomorphism at $q$ and $p$ respectively and we have nothing to prove. So we assume that $q\in Z$.
	
	First suppose that $q$ is a 0--point and $Z$ is a (0,0)--SV for $E$ at $q$. Then by Lemma \ref{QTF0}, $p_1\in (\varphi\lambda)^{-1}(q)$ is also a 0--point for $D_1=\lambda^*(D)_{\mathrm{red}}$, and there exist regular parameters $\bar y_1,\dots,\bar y_m$ in $\mathcal O_{V,q}$ such that $\bar y_1=\cdots=\bar y_c=0$ are local equations of $Z$ at $q$, and there exist regular system of parameters $x_1,\dots,x_d$ at $p_1$, and $\beta_1,\dots,\beta_{c-1}\in\mathcal K$ such that 
	\begin{equation}
	\bar y_i=\begin{cases}
	x_c(x_i+\beta_i) & \mbox{ for } 1\leqslant i\leqslant c-1, \\
	x_i & \mbox{ for } c\leqslant i\leqslant m,
	\end{cases}
	\end{equation}
	and $\mathcal I_Z\mathcal O_{U_1,p_1}=\langle x_c\rangle$, since $\lambda:U_1\to U$ is a SPS of $Z$, i.e., $W_Z(U_1)=\emptyset$. Let $q_1\in\pi^{-1}(q)$ be the closed point corresponding to the maximal ideal
	$$\mathfrak m:=\langle \frac{\bar y_1}{\bar y_c}-\beta_1,\dots,\frac{\bar y_{c-1}}{\bar y_c}-\beta_{c-1},\bar y_c,\dots,\bar y_m \rangle\mathcal O_{V_1}.$$
	We then define $\phi_1(p_1):=q_1$ which is a 0--point for $E_1$ since $q\notin E$. In this way, we have that 
	\begin{equation}
	y_i :=\begin{cases}
	\frac{\bar y_i}{\bar y_c}-\beta_i & \mbox{ for } 1\leqslant i\leqslant c-1, \\
	\bar y_i & \mbox{ for } c\leqslant i\leqslant m,
	\end{cases}
	\end{equation}
	are regular parameters in $\mathcal O_{V_1,q_1}$ such that $y_i=x_i$ for $1\leqslant i\leqslant m$, which means that $\phi_1$ is smooth at $p_1$.
	
	Now suppose that $q$ is an $\ell$--point and $Z$ is an $(\ell,\bar\ell)$--SV for $E$ at $q$, for some $\ell> 0$ and $0\leqslant \bar\ell \leqslant \ell$. Then $p_1\in (\varphi\lambda)^{-1}(q)$ is an $n$-point for $D_1$, for some $n>0$, and by Lemma \ref{QTF}, there exist permissible parameters $x_1,\dots,x_d$ at $p_1$ on $D_1$, and permissible parameters $\bar y_1,\dots,\bar y_m$ for $Z$ at $q$ on $E$ such that 
	$$
	\bar y_1=\cdots=\bar y_{\bar\ell}=\bar y_{\ell+1}=\cdots=\bar y_{\ell+s}=0
	$$
	are local equations of $Z$ at $q$, where $s:=c-\bar\ell\geqslant 0$, and there exist a matrix $\mathbf{a}=(a_{ij})_{\substack{1\leqslant i\leqslant \ell+s \\ 1\leqslant j\leqslant n}}$ with natural entries, unit series $\delta_1,\dots,\delta_{\ell+s}\in\hat{\mathcal O}_{U_1,p_1}$, and $\beta_{n+1},\dots,\beta_{n+s}\in\mathcal K$ such that one of the quasi-toroidal forms (\ref{QTF1}) or (\ref{QTF2}) which is principal holds for $p_1$, since $W_Z(U_1)=\emptyset$. Therefore, after possibly permuting $\bar y_1,\dots,\bar y_{\bar\ell}$, and $\bar y_{\ell+1},\dots,\bar y_{\ell+s}$, there are three possibilities to consider:
	\begin{enumerate}
		\item\label{item1} Quasi-toroidal form (\ref{QTF1}) holds for $p_1$, and for all $j\in [n]$, $\min\{a_{ij}\}_{i\in I}=a_{1j}$ where $I:= [\bar\ell]\cup \{\ell+1,\dots,\ell+s\}$, so that $\mathcal I_Z\hat{\mathcal O}_{U_1,p_1}=\langle  \delta_1 x_1^{a_{11}} \cdots  x_n^{a_{1n}} \rangle$. In addition, if $s>0$, for all $j\in [n]$, $a_{(\ell+1)j}=\cdots=a_{(\ell+s)j}=\min\{a_{ij}\}_{i\in I}=a_{1j}$.
		
		\item Quasi-toroidal form (\ref{QTF1}) holds for $p_1$, and for some $j\in[n]$, $a_{\ell+1,j}\neq 0$, and $\beta_{n+1}\in\mathcal K^{\times}$, so that $\mathcal I_Z\hat{\mathcal O}_{U_1,p_1}=\langle \delta_{\l+1} x_1^{a_{\l+1,1}} \cdots  x_n^{a_{\l+1,n}} (x_{n+1}+\beta_{n+1})\rangle$. We note that in this case, we must have $s>0$, and for all $j\in [n]$, $a_{(\ell+1)j}=\cdots=a_{(\ell+s)j}=\min\{a_{ij}\}_{i\in I}$.
		
		\item Quasi-toroidal form (\ref{QTF2}) holds for $p_1$, and $\mathcal I_Z\hat{\mathcal O}_{U_1,p_1}=\langle \delta_{\l+1} x_1^{a_{\l+1,1}} \cdots  x_n^{a_{\l+1,n}} \rangle$.
	\end{enumerate}
	We will complete the proof when (\ref{item1}) holds for $p_1$. The proof of the other cases are similar and is left to the reader.
	
	Suppose that (\ref{QTF1}) holds for $p_1$, i.e., 
	\begin{equation}
	\bar y_i=\left\{
	\begin{array}{ll}
	\delta_i x_1^{a_{i1}} \cdots  x_n^{a_{in}}, & 1\leqslant i\leqslant \ell \\
	\delta_i x_1^{a_{i1}} \cdots  x_n^{a_{in}}(x_{n-\ell+i}+\beta_{n-\ell+i}), & \ell+1\leqslant i\leqslant \ell+s \\
	\ \ \ x_{n-\ell+i}, & \ell+s+1\leqslant i\leqslant m,
	\end{array}
	\right.
	\end{equation}
	and for all $j\in [n]$, $\min\{a_{ij}\}_{i\in I}=a_{1j}$ where $I:= [\bar\ell]\cup \{\ell+1,\dots,\ell+s\}$, so that 
	$$
	\mathcal I_Z\hat{\mathcal O}_{U_1,p_1}=\langle  x_1^{a_{11}} \cdots  x_n^{a_{1n}} \rangle.
	$$
	In addition, if $s>0$, for all $j\in [n]$, $a_{(\ell+1)j}=\cdots=a_{(\ell+s)j}=\min\{a_{ij}\}_{i\in I}=a_{1j}$. Let
	$$
	\overline{\mathbf a}=\left(\begin{array}{lcl}
	a_{11} & \cdots & a_{1n}\\
	a_{21}-a_{11} & \cdots & a_{2n}-a_{1n}\\
	\vdots & &\vdots\\
	a_{\bar\ell 1}-a_{11} & \cdots & a_{\bar\ell n}-a_{1n}\\
	a_{\bar\ell+1,1} & \cdots & a_{\bar\ell+1,n}\\
	\vdots & &\vdots\\
	a_{\ell1} & \cdots & a_{\ell n}\\
	\end{array}\right)\in\mathbb N^{\ell\times n}.
	$$
	We first observe that $\overline{\mathbf a}$ has no zero column. If there exists $j_0\in[n]$ such that 
	$$
	a_{1j_0}=a_{2j_0}-a_{1j_0}=\cdots=a_{\bar\ell j_0}-a_{1j_0}=a_{\bar\ell+1, j_0}=\cdots=a_{\ell j_0}=0,
	$$ 
	then we must have $a_{1j_0}=a_{2j_0}=\dots=a_{\ell j_0}=0$, which contradicts $\sum_{i=1}^{\ell}a_{ij}\neq 0$ for all $j\in [n]$. Recall we also have $\sum_{j=1}^{n}a_{ij}\neq 0$ for all $i\in [\ell]$. However, it is possible that for some $i$, $2\leqslant i\leqslant \bar\ell$, the $i$-th row of $\overline{\mathbf a}$ is zero. After possibly permuting $\bar y_2,\dots,\bar y_{\bar\ell}$, we may assume that there exists $t$ such that for all $i$, $t+1\leqslant i\leqslant \bar\ell$, $a_{i1}-a_{11}=\cdots=a_{in}-a_{1n}=0$. (We note that the number $\bar\ell-t$ of zero rows of $\overline{\mathbf a}$ is at most $\min\{\ell-r,\bar\ell-1\}$, where $r$ is ${\mathrm {rank}}\:\overline{\mathbf a}={\mathrm    {rank}}\:(a_{ij})_{\substack{1\leqslant i\leqslant \ell \\ 1\leqslant j\leqslant n}}$). We set
	\begin{align*}
	\tilde\delta_i\ \quad\quad\qquad &:= \delta_1^{-1}\delta_i \quad\qquad\qquad\qquad\qquad\qquad\qquad\mbox{ \ for \ } 2\leqslant i\leqslant \ell+s, \\
	\tilde\beta_{n+m-r-t+i} &:= \tilde\delta_{i}(p_1)\quad\qquad\qquad\qquad\qquad\qquad\qquad\mbox{ for \ } t+1\leqslant i\leqslant \bar\ell,\\  
	\tilde x_{n+m-r-t+i} &:= \tilde\delta_{i}-\tilde\delta_{i}(p_1)\quad\quad\qquad\qquad\qquad\quad\qquad\mbox{ for \ } t+1\leqslant i\leqslant \bar\ell,\\   
	\tilde\beta_{n-\ell+i}\ \quad\quad &:= \tilde\delta_{n-\ell+i}(p_1)\beta_{n-\ell+i}\quad\qquad\qquad\quad\qquad\mbox{ \,for \ } \ell+1\leqslant i\leqslant \ell+s,\\
	\tilde x_{n-\ell+i}\ \quad\quad &:=  \tilde\delta_{n-\ell+i}(x_{n-\ell+i}+\beta_{n-\ell+i})-\tilde\beta_{n-\ell+i}\quad\mbox{\,for \ } \ell+1\leqslant i\leqslant \ell+s,
	\end{align*}
	where $\tilde\delta_j(p_1)\in\mathcal K^{\times}$ denotes the unit series $\tilde\delta_j$ modulo the maximal ideal of $\hat{\mathcal O}_{U_1,p_1}$, for $j$, $t+1\leqslant j\leqslant \bar\ell$ or $n+1\leqslant j\leqslant n+s$. We also note that for $j$, $n+1\leqslant j\leqslant n+s$, $\tilde\beta_j\in\mathcal K$ since $\beta_j\in\mathcal K$, and for $j$, $n+m-r+1\leqslant j\leqslant n+m-r+(\bar\ell-t)$, $\tilde\beta_j\in\mathcal K^{\times}$.
	
	We reindex $x_j$'s and $\tilde\beta_j$'s by defining the map $\theta:[d]\to[d]$ as follows:
	\begin{equation}
	\theta(j)=\begin{cases}
	j & \hbox{for } j\in[n], \hbox{ or }  n+(\bar\ell-t)+(m-r)< j \leqslant d,\\
	j+m-r & \hbox{for } n< j \leqslant n+(\bar\ell-t), \\
	j-\bar\ell+t & \hbox{for } n+(\bar\ell-t)< j \leqslant n+(\bar\ell-t)+(m-r).
	\end{cases}
	\end{equation}
	Let 
	$$
	{\rm x}_j:=\begin{cases}\tilde x_{\theta(j)}& \hbox{for }n+1\leqslant j \leqslant n+\bar\ell-t+s, \\ x_{\theta(j)} & \hbox{otherwise},\end{cases}
	$$
	and let $\beta'_j:=\tilde\beta_{\theta(j)}$ for $j$, $n+1\leqslant j \leqslant n+\bar\ell-t+s$. We then define
	\begin{equation}\label{almost TF1}
	y_i:=
	\begin{cases}
	\bar y_1\ \,\quad\quad\qquad\qquad=\delta_1 \x_1^{a_{11}} \cdots  \x_n^{a_{1n}}, & \mbox{ for \ } i=1, \\
	\frac{\bar y_i}{\bar y_1}\ \,\quad\quad\qquad\qquad=\tilde\delta_i \x_1^{a_{i1}-a_{11}} \cdots  \x_n^{a_{in}-a_{1n}}, & \mbox{ for \ } 2\leqslant i\leqslant t, \\
	\frac{\bar y_i}{\bar y_1}-\beta'_{n-t+i}\ \,\quad\quad=\x_{n-t+i}, & \mbox{ for \  } t+1\leqslant i\leqslant \bar\ell, \\ 
	\bar y_i\ \ \,\quad\quad\qquad\qquad=\tilde\delta_i \x_1^{a_{i1}} \cdots  \x_n^{a_{in}}, & \mbox{ for \ } \bar\ell+1\leqslant i\leqslant \ell, \\
	\frac{\bar y_i}{\bar y_1}-\beta'_{n-(\ell-\bar\ell+t)+i}=\x_{n-(\ell-\bar\ell+t)+i}, & \mbox{ for \  } \ell+1\leqslant i\leqslant \ell+s, \\
	\bar y_i \ \,\quad\quad\qquad\qquad\ =\x_{n-(\ell-\bar\ell+t)+i}, & \mbox{ for \ }\ell+s+1\leqslant i\leqslant m.
	\end{cases}
	\end{equation}
	Let $q_1\in\pi^{-1}(q)$ be the closed point corresponding to the maximal ideal
	$\langle y_1,\dots,y_m \rangle\hat{\mathcal O}_{V_1}$.
	We define $\phi_1(p_1):=q_1$, and we note that 
	$$y_1^{\bar\ell} y_2\cdots y_t \times \prod_{i=t+1}^{\bar\ell}(y_i+\beta'_{n-t+i})\times y_{\bar\ell+1}\cdots y_{\ell}=0 $$
	is a local equation of $\pi^*(E)$ at $q_1$ where $\beta'_{n-t+i}\in\mathcal K^{\times}$ for $i$, $t+1\leqslant i\leqslant \bar\ell$. Thus writing $\ell_1:=\ell-\bar\ell+t$, $q_1$ is an $\ell_1$--point for $E_1=\pi^*(E)_{\rm{red}}$, and $y_1\cdots y_t y_{\bar\ell+1}\cdots y_{\ell}=0$ is a local equation of $E_1$ at $q_1$. We then reindex $y_i$'s as follows. Consider the map $\sigma:[m]\to[m]$ defined by
	\begin{equation}
	\sigma(i)=\begin{cases}
	i & \hbox{for } i\in[t]\cup([m]\setminus[\ell]), \\
	i+\bar\ell-t & \hbox{for } t< i\leqslant \ell_1, \\
	i+\bar\ell-\ell & \hbox{for } \ell_1< i\leqslant \ell, 
	\end{cases}
	\end{equation}
	and let ${\rm y}_i:=y_{\sigma(i)}$ for $i\in [m]$. We note that $\y_1\cdots\y_{\ell_1}=0$ is a local equation of $E_1$ at $p_1$. We also define
	\begin{equation}
	b_{ij} :=\begin{cases}
	a_{ij} & \hbox{for } i=1\ ,\ j\in[n], \\
	a_{ij}-a_{1j} & \hbox{for } 2\leqslant i\leqslant t\ ,\ j\in[n], \\
	a_{\sigma(i),j} & \hbox{for } t< i\leqslant \ell_1\ ,\ j\in[n],
	\end{cases} \mbox{ and } \delta'_i :=\begin{cases}
	\delta_i & \hbox{for } i=1, \\
	\tilde\delta_i & \hbox{for } 2\leqslant i\leqslant t, \\
	\tilde\delta_{\sigma(i)} & \hbox{for } t< i\leqslant \ell_1.
	\end{cases}
	\end{equation}
	In this way, we obtain permissible parameters $\y_1,\dots,\y_m$ at the $\ell_1$--point $q_1$ on $E_1$, permissible parameters $\x_1,\dots,\x_d$ at the $n$--point $p_1$ on $D_1$ such that
	\begin{equation}
	\y_i=\left\{
	\begin{array}{ll}
	\delta'_i \x_1^{b_{i1}} \cdots  \x_n^{b_{in}} & \mbox{ for }1\leqslant i\leqslant \ell_1, \\
	\ \ \ \x_{n-\ell_1+i} & \mbox{ for } \ell_1< i\leqslant m,
	\end{array}
	\right.
	\end{equation}
	where $(b_{ij})_{\substack{1\leqslant i\leqslant \ell_1 \\ 1\leqslant j\leqslant n}}\in\mathbb N^{\ell_1\times n}$ has the property that for all $j\in [n]$, $\sum_{i=1}^{\ell_1}b_{ij}> 0$, and for all $i\in[\ell_1]$, $\sum_{j=1}^{n}b_{ij}> 0$. It is also clear that the variables $\x_1,\dots,\x_{n+(m-\l_1)}$ do not appear in $\delta'_i \in \hat{\mathcal O}_{U_1,p_1}$, for $1\leqslant i\leqslant \ell_1$. Thus $\phi_1$ is toroidal at $p_1$, and by our definition $\phi_1(p_1)=\pi^{-1}\varphi\lambda(p_1)$.
\end{proof}

%%%%%%%%%%%%%%%%%%%%%%%%%%%%%%%%%%%%%%%%%%%%%%%%%%%%%%%%%%%%%%%%%%%%%%

\begin{Theorem}\label{Locally Toroidalization}
	Suppose that $\varphi:(X,U_{\a},D_{\a})_{\a\in I}\to (Y,V_{\a},E_{\a})_{\a\in I}$ is a locally toroidal morphism of nonsingular varieties, and let $\mathcal{E}_0=\sum_{\alpha \in I}\overline{E}_{\alpha}$, where $\overline{E}_{\alpha}$ is the Zariski closure of $E_{\alpha}$ in $Y$. There exists a commutative diagram
	\begin{equation}\label{multidiagram}
	\xymatrix{
		\lambda: \widetilde X=X_{n_0} \ar[r]^{\lambda_{n_0}}\ar[d]_{\tilde\varphi=\phi_{n_0}} & X_{n_0-1}\ar[r]\ar[d]_{\phi_{n_0-1}} & \cdots\ar[r] & X_1\ar[r]^{\lambda_1}\ar[d]^{\phi_1} & X=X_0\ar[d]^{\phi_0=\varphi} \\
		\pi: \widetilde Y=Y_{n_0} \ar[r]_{\pi_{n_0}} & Y_{n_0-1}\ar[r] & \cdots\ar[r] & Y_1\ar[r]_{\pi_1} & Y=Y_0}
	\end{equation}
	with the following properties. For $\a\in I$ and $0\leqslant k\leqslant n_0$, let
	\begin{center}
		\begin{tabular}{lll}
			
			$\Pi_k=\pi_0\circ\cdots\circ\pi_k$, & \ \  & $\Lambda_k=\lambda_0\circ\cdots\circ\lambda_k$, \\
			
			$V_{k,\alpha}=\Pi_k^{-1}(V_{\alpha})$, & \ \  & $U_{k,\alpha}=\Lambda_k^{-1}(U_{\alpha})$, \\ 
			
			$\pi_{k,\alpha}=\pi_k|_{V_{k,\alpha}}:V_{k,\alpha}\rightarrow V_{k-1,\alpha}$, & \ \  & $\lambda_{k,\alpha}=\lambda_k|_{U_{k,\alpha}}:U_{k,\alpha}\rightarrow U_{k-1,\alpha}$, \\
			
			$\Pi_{k,\alpha}=\Pi_k|_{V_{k,\alpha}}:V_{k,\alpha}\to V_{\alpha}$, & \ \  & $\Lambda_{k,\alpha}=\Lambda_k|_{U_{k,\alpha}}:U_{k,\alpha}\to U_{\alpha}$, \\ 
			
			$E_{k,\alpha}=\Pi_{k,\alpha}^{-1}(E_{\alpha})=\Pi_{k,\alpha}^*(E_\alpha)_{\rm red}$, & \ \ & $D_{k,\alpha}=\Lambda_{k,\alpha}^{-1}(D_{\alpha})=\Lambda_{k,\alpha}^*(D_\alpha)_{\rm red}$.
		\end{tabular}
	\end{center}
	\begin{enumerate}
		
		\item The morphisms $\lambda:\widetilde X\to X$ and $\pi:\widetilde Y\to Y$ are sequences of monoidal transforms.
		
		\item  For all $\a\in I$ and $0\leqslant k\leqslant n_0$, $D_{k,\a}$ is a SNC divisor on $U_{k,\a}$, $E_{k,\a}$ is a SNC divisor on $V_{k,\a}$, and $\phi_k:(X_k,U_{k,\a},D_{k,\a})_{\a\in I}\to (Y_k,V_{k,\a},E_{k,\a})_{\a\in I}$ is a locally toroidal morphism of nonsingular varieties.
		
		\item The divisor $\widetilde{\mathcal{E}}:=\pi^{-1}(\mathcal{E}_0)$ is  SNC on $\widetilde Y$, and for all $\a \in I$, $\pi^{-1}(E_{\a})\subset\widetilde{\mathcal{E}}$.		
	\end{enumerate}
\end{Theorem}

\begin{proof}
	We will construct the diagram (\ref{multidiagram}) inductively using embedded resolution of singularities and strong principalization of ideal sheaves. The required properties follows from Theorem \ref{ERS} and Lemma \ref{tor after principalization}. Let
	\begin{equation}\label{ERS'}
	\pi:\widetilde Y = Y_{n_0}\xrightarrow{\pi_{n_0}} Y_{n_0-1}\xrightarrow{}\cdots\xrightarrow{}Y_k\xrightarrow{\pi_k}Y_{k-1}\xrightarrow{}\cdots\to Y_1\xrightarrow{\pi_1}Y_0=Y
	\end{equation}
	be the ERS of $\mathcal E_0$ so that $\widetilde{\mathcal{E}}:=\pi^{-1}(\mathcal{E}_0)$ is a SNC divisor on $\widetilde Y$. By Theorem \ref{ERS}, the ERS sequence (\ref{ERS'}) has the following properties:
	\begin{enumerate}
		\item\label{ERS1} For all $\a\in I$ and $0\leqslant k\leqslant n_0$, $E_{k,\alpha}$ is a SNC divisor on $V_{k,\alpha}$, and $Z_{k,\alpha}:=Z_k\cap V_{k,\a}$ makes SNCs with $E_{k,\alpha}$ on $V_{k,\alpha}$, where $Z_k$ denotes the center of $\pi_{k+1}:Y_{k+1}\to Y_k$.\footnote{It is possible that $Z_{k,\alpha}\cap E_{k,\alpha}\ne \emptyset$ but
			$Z_{k,\alpha}\not\subset E_{k,\alpha}$.}
		\item\label{ERS2} 
		For all $0\leqslant k\leqslant n_0$, $(\sum_{\alpha\in I}\overline E_{k,\alpha})_{\rm red}\subseteq\Pi_k^{-1}(\widetilde{\mathcal E}_0)$ which implies that for all $\a\in I$, $E_{k,\a}\subset\Pi_k^{-1}(\widetilde{\mathcal E}_0)$. In particular, for all $\a \in I$, we have $\pi^{-1}(E_{\a})\subset\widetilde{\mathcal{E}}$.	 
	\end{enumerate}
	Suppose that we have constructed the commutative diagram 
	\begin{equation}
	\xymatrix{
		X_{k}\ar[r]\ar[d]_{\phi_{k}} & \cdots\ar[r] & X_1\ar[r]^{\lambda_1}\ar[d]^{\phi_1} & X_0\ar[d]^{\phi_0} \\
		Y_{k}\ar[r] & \cdots\ar[r] & Y_1\ar[r]_{\pi_1} & Y_0}
	\end{equation}
	for which the conclusions of the theorem hold. After the blowup $\pi_{k+1}:Y_{k+1}\to Y_k$, we obtain the rational map $\pi_{k+1}^{-1}\circ\phi_k:X_k \dashrightarrow Y_{k+1}$. To resolve the indeterminacy of $\pi_{k+1}^{-1}\circ\phi_k$, which is $W_{Z_k}(X_k)=\{p\in X_k\ |\ \mathcal I_{Z_k}\mathcal{O}_{X_k,p} \mbox{\ \ is not principal}\}$, we let 
	\begin{equation}\label{SPS1}
	\lambda_{k+1}: X_{k+1}=X_k^{n_1}\xrightarrow{}X_k^{n_1-1}\to\cdots\to X_k^i\xrightarrow{\lambda_{k+1}^i}X_k^{i-1}\to\cdots\to X_k^1\xrightarrow{}X^0_k=X_k
	\end{equation}
	be a SPS of the ideal sheaf $\mathcal N_k$ of $W_{Z_k}(X_k)$. Clearly, (\ref{SPS1}) induces a SPS 
	$$
	\lambda_{k+1,\a}:=\lambda_{k+1}|_{U_{k+1,\a}}:U_{k+1,\a}=U^{n_1}_{k+1,\a}\to U^{n_1-1}_{k+1,\a}\to \cdots \to U_{k,\a}^1\to U^0_{k,\a}=U_{k,\a}
	$$
	of $W_{Z_{k,\a}}(U_{k,\a})=W_{Z_k}(X_k)\cap U_{k,\a}$ for each $\a\in I$.
	% We now prove that $\phi_{k+1}$ is locally toroidal. Let $p_1\in X_{k+1}$ lie in $U_{k+1,\a}$ for some $\a\in I$, and write $q:=\phi_k\circ\lambda_{k+1}(p_1)$. If $q\not\in Z_{k,\a}$, then $p:=\lambda_{k+1}(p_1)\not\in\phi_k^{-1}(Z_{k,\a}$), and hence $p\not\in W_{Z_{k,\a}}(U_{k,\a})$ since $W_{Z_{k,\a}}(U_{k,\a})\subseteq\phi_k^{-1}(Z_{k,\a})$. In this case, $\lambda_{k+1}$ and $\pi_{k+1}$ are isomorphism at $p_1$ and $q_1:=\phi_{k+1}(p_1)$ respectively, and we are done since $\phi_k|_{U_{k,\a}}$ is toroidal by the induction hypothesis. We then assume that $q\in Z_{k,\a}$.
	%
	%Let $q\in Z_{k,\a}$ be an $\l$--point, $\l\geqslant 0$, for $E_{k,\a}$ and suppose that $Z_{k,\a}$ is an $(\l,\bar\l)$--SV, $0\leqslant\bar\l\leqslant\l$, for $E_{k,\a}$ at $q$, recalling that $Z_{k,\a}$ makes SNCs with $E_{k,\a}$ by (\ref{ERS1}).
	%
	%
	%
	%
	%To prove that $\phi_{k+1}$ is locally toroidal, 
	We also note that by the induction hypothesis, $\phi_k:(X_k,U_{k,\a},D_{k,\a})_{\a\in I}\to (Y_k,V_{k,\a},E_{k,\a})_{\a\in I}$ is a locally toroidal morphism, i.e., for all $\a$, $D_{k,\a}\subset U_{k,\a}$ and $E_{k,\a}\subset V_{k,\a}$ are SNC divisors, and $$
	\phi_{k,\a}:=\phi_k|_{U_{k,\a}}:(U_{k,\a},D_{k,\a})\to (V_{k,\a},E_{k,\a})
	$$
	is toroidal. In addition, (\ref{ERS1}) tells us that $Z_{k,\a}$ makes SNCs with $E_{k,\a}$ for all $\a\in I$. Applying Lemma \ref{tor after principalization} to each $\a\in I$, we see that $D_{k+1,\a}\subset U_{k+1,\a}$ and $E_{k+1,\a}\subset V_{k+1,\a}$ are SNC divisors, for all $\a\in I$, and
	$$
	\phi_{k+1,\a}:=\pi_{k+1,\a}^{-1}\circ\phi_{k,\a}\circ\lambda_{k+1,\a}:(U_{k+1,\a},D_{k+1,\a})\to (V_{k+1,\a},E_{k+1,\a})
	$$
	is toroidal. The morphism $\phi_{k+1}:X_{k+1}\to Y_{k+1}$ defined by $\phi_{k+1}|_{U_{k+1,\a}}:=\phi_{k+1,\a}$ is obviously well-defined on the intersections of every pair of sets in the open cover $\{U_{k+1,\a}\}_{\a\in I}$, and 
	$$
	\phi_{k+1}:(X_{k+1},U_{k+1,\a},D_{k+1,\a})_{\a\in I}\to (Y_{k+1},V_{k+1,\a},E_{k+1,\a})_{\a\in I}
	$$
	is locally toroidal by construction. Repeating this process results in the diagram (\ref{multidiagram}) for which the conclusions of the theorem hold.
\end{proof}

\begin{proof}[Proof of Theorem \ref{Main Theorem}]
	Let $\mathcal{E}_0=\sum_{\alpha \in I}\overline{E}_{\alpha}$, where $\overline{E}_{\alpha}$ is the Zariski closure of $E_{\alpha}$ in $Y$, and suppose that
	$$\begin{CD}
	\widetilde{X}@>\widetilde\varphi>>\widetilde{Y}\\
	@V\lambda VV @VV\pi V\\ X@>>\varphi>Y
	\end{CD}$$
	is the commutative diagram constructed in Theorem \ref{Locally Toroidalization}. For $\a\in I$, let $\widetilde U_{\a}=\lambda^{-1}(U_{\a})$, $\widetilde D_{\a}=\lambda^{-1}(D_{\a})$, and  $\widetilde V_{\a}=\pi^{-1}(V_{\a})$, $\widetilde E_{\a}=\pi^{-1}(E_{\a})$. We will prove that the locally toroidal morphism $$\widetilde\varphi:(\widetilde X,\widetilde U_{\a},\widetilde D_{\a})_{\a\in I} \to (\widetilde Y,\widetilde V_{\a},\widetilde E_{\a})_{\a\in I}$$ is in fact toroidal with respect to $\widetilde{E}=\pi^{-1}(\mathcal{E}_0)$ and $\widetilde{D}=\widetilde\varphi^{-1}(\widetilde{E})$.
	
	First, we simply observe that ${\rm Sing}\,\widetilde\varphi\subseteq \widetilde D$. If $p\in \widetilde X\setminus\widetilde D$, then $q:=\widetilde\varphi(p)\in \widetilde Y\setminus\widetilde E$, and hence $q\notin \widetilde E_{\a}$ for all $\a\in I$, since $\widetilde E_{\a}\subset\widetilde E$ for all $\a\in I$, by Theorem \ref{ERS}. Thus for all $\a\in I$, $p\notin\widetilde\varphi^{-1}(\widetilde E_{\a})$, and in particular, $p\notin\widetilde D_{\a}=\widetilde\varphi|_{\widetilde U_{\a}}^{-1}(\widetilde E_{\a})$ which implies that $\widetilde\varphi$ is smooth at $p$, since $\widetilde\varphi|_{U_{\a}}:(\widetilde U_{\a},\widetilde D_{\a})\to (\widetilde V_{\a},\widetilde E_{\a})$ is toroidal and ${\rm Sing}\,\widetilde\varphi|_{U_{\a}}\subseteq \widetilde D_{\a}$.
	
	Now suppose that $p\in\widetilde D\cap U_{\a}$ for some $\a\in I$, so that $q=\widetilde\varphi(p)=\widetilde\varphi_{\a}(p)\in\widetilde E$. Assume that $q$ is a $k$--point on $\widetilde E_{\a}$, and it is an $\l$--point on $\widetilde E$. We note that $\l>0$ and $0\leqslant k\leqslant \l$. If $k=0$, i.e., $q\notin\widetilde E_{\a}$, then $\widetilde\varphi_{\a}$ is smooth at $p$, and so is $\widetilde\varphi$. Hence $\widetilde D$ is a SNC divisor at $p$ since $\widetilde E$ is a SNC divisor at $q$. If $k=\l$, then $\mathcal I_{\widetilde E_{\a},q}=\mathcal I_{\widetilde E,q}$ and thus $\mathcal I_{\tilde D_{\a},p}=\mathcal I_{\tilde D,p}$. It is clear in this case that $\widetilde\varphi:(\widetilde X,\widetilde D)\to (\widetilde Y,\widetilde E)$ is toroidal at $p$ because $\widetilde\varphi_{\a}:(\widetilde U_{\a},\widetilde D_{\a})\to (\widetilde V_{\a},\widetilde E_{\a})$ is toroidal at $p$. 
	
	The remaining case to consider is when $q$ is a $k$--point on $\widetilde E_{\a}$, and it is an $\l$--point on $\widetilde E$ and $0<k<\l$. Applying theorem \ref{TF} to $\widetilde\varphi_{\a}$, if $p\in\widetilde\varphi_{\a}^{-1}(q)$ is an $n$--point on $\widetilde D_{\a}$, then we have permissible parameters $y_1,\dots,y_m$ at $q$ on $\widetilde E_{\a}$, and  permissible parameters $x_1,\dots,x_d$ at $p$ on $\widetilde D_{\a}$ such that
	\begin{equation}\label{eq111}
	y_i=\left\{
	\begin{array}{ll}
	\delta_i x_1^{a_{i1}} \cdots  x_n^{a_{in}}, & 1\leqslant i\leqslant k \\
	\ \ \ x_{n-k+i}, & k< i\leqslant m
	\end{array}
	\right.
	\end{equation}
	where  $\delta_i$, $1\leqslant i\leqslant k$, are unit series in which the variables $x_1,\dots,x_{n+(m-k)}$ do not appear, and $(a_{ij})_{\substack{1\leqslant i\leqslant k \\ 1\leqslant j\leqslant n}}\in\mathbb N^{k\times n}$ satisfies:
	\begin{equation}\label{ex matrix condition?}
	\mbox{for all }j\in [n], \sum_{i=1}^{k}a_{ij}> 0,\mbox{ and for all }i\in[k], \sum_{j=1}^{n}a_{ij}> 0.
	\end{equation}
	We may if necessary change the parameters $y_{k+1},\dots,y_m$, and we let $\tilde y_i=y_i$ for $i\in[k]$, to obtain permissible parameters $\mathbf{\tilde y}=(\tilde y_1,\dots,\tilde y_m)$ at $p$ on $\widetilde E$, which satisfies $\det J(\mathbf{\tilde y};\mathbf{y})\neq 0$ by the formal inverse function theorem, and $\tilde y_1\cdots\tilde y_{\l}=0$ is a local equation of $\widetilde E$ at $q$. Writing $N:=n+\l-k$, we also set 
	$$
	\tilde x_j:=\begin{cases}
	x_i & \mbox{ for }1\leqslant j\leqslant n\\
	\tilde y_{k-n+j} & \mbox{ for }n< j\leqslant n+m-k\\
	x_j & \mbox{ for }n+m-k< j\leqslant d
	\end{cases}
	\mbox{ and } \mathbf{b}=(b_{ij})_{\substack{1\leqslant i\leqslant \l \\ 1\leqslant j\leqslant N}}:= \left(
	\begin{array}{c|c}
	(a_{ij})_{\substack{1\leqslant i\leqslant k \\ 1\leqslant j\leqslant n}} & \mathrm{O}_{k\times \l-k} \\
	\hdotsfor[.5]{2}\\
	\mathrm{O}_{\l-k\times n} & I_{\l-k}
	\end{array}
	\right),   
	$$
	and we obtain 
	\begin{equation}
	\tilde y_i=\left\{
	\begin{array}{ll}
	\delta_i \tilde x_1^{b_{i1}} \cdots  \tilde x_N^{b_{iN}}, & 1\leqslant i\leqslant \l \\
	\ \ \ \tilde x_{N-\l+i}, & \l< i\leqslant m.
	\end{array}
	\right.
	\end{equation}
	We note that $p$ is an $N$--point on $\widetilde D=\widetilde\varphi^{-1}(\widetilde E)$ and $\tilde x_1\cdots\tilde x_N=0$ is a local equation of $\widetilde D$ at $p$. It is also clear from (\ref{ex matrix condition?}) that $\mathbf{b}$ has the property that for all $j\in [N]$, $\sum_{i=1}^{\l}b_{ij}> 0$, and for all $i\in[\l]$, $\sum_{j=1}^{N}b_{ij}> 0$. Hence, to complete the proof, we only need to show that $\det J(\mathbf{\tilde x};\mathbf{x})\neq 0$ which shows that $\mathbf{\tilde x}=(\tilde x_1,\dots,\tilde x_d)$ are permissible parameters at $p$ on $\widetilde D$. 
	
	Suppose that, for $k+1\leqslant i\leqslant m$, 
	\begin{equation}\label{expansion of tilde y?}
	\tilde y_i=\sum_{g=1}^m \gamma_{ig}y_g + \mbox{ higher degree terms in }y_g\mbox{'s}
	\end{equation}
	for some $\gamma_{ig}\in\mathcal K$. After substituting (\ref{eq111}) in (\ref{expansion of tilde y?}), for $k+1\leqslant i\leqslant m$, we obtain
	\begin{equation}\label{eq112}
	\tilde x_{n-k+i} =\tilde y_i=\sum_{g=1}^k \gamma_{ig}\delta_g x_1^{a_{g1}} \cdots  x_n^{a_{gn}} +\sum_{g=k+1}^m \gamma_{ig}x_{n-k+g} + \mbox{ higher degree terms in }x_j\mbox{'s}.
	\end{equation}
	We will observe that the Jacobian matrix $J(\mathbf{\tilde x};\mathbf{x})=$
	$$
	\left(
	\begin{array}{c|c|c|c}
	I_{n} & \mathrm{O} & \mathrm{O}& \mathrm{O} \\
	\hdotsfor[.5]{4}\\
	\left(\frac{\partial\tilde x_{n-k+g}}{\partial x_j}\right)_{\substack{k < g\leqslant m \\ 1\leqslant j\leqslant n}} &
	\left(\frac{\partial\tilde x_{n-k+g}}{\partial x_j}\right)_{\substack{k < g\leqslant m \\ n < j\leqslant n-k+m}} & \left(\frac{\partial\tilde x_{n-k+g}}{\partial x_j}\right)_{\substack{k < g\leqslant m \\ n-k+m < j\leqslant n-r+m}} & \mathrm{O}  \\
	\hdotsfor[.5]{4}\\
	\mathrm{O} & \mathrm{O} & I_{k-r}  & \mathrm{O}\\
	\hdotsfor[.5]{4}\\
	\mathrm{O} & \mathrm{O}  & \mathrm{O} & I_{d-(n+m-r)} \\
	\end{array}
	\right)
	$$
	is full rank where $r={\rm rank}\,(a_{ij})_{\substack{1\leqslant i\leqslant k \\ 1\leqslant j\leqslant n}}$. An easy computation using (\ref{eq112}) shows that 
	$$
	\left(\frac{\partial\tilde x_{n-k+g}}{\partial x_j}\right)_{\substack{k < g \leqslant m \\ n < j\leqslant n-k+m}}=(\gamma_{ig})_{\substack{k < i\leqslant m \\ k < g\leqslant m}},
	$$
	and thus after performing elementary row operations we obtain
	$$
	J(\mathbf{\tilde x};\mathbf{x})\leftrightarrow\left(
	\begin{array}{c|c|c}
	I_{n} & \mathrm{O} & \mathrm{O} \\
	\hdotsfor[.5]{3}\\
	\mathrm{O} &(\gamma_{ig})_{\substack{k < i \leqslant m \\k < g \leqslant m}} & \mathrm{O} \\
	\hdotsfor[.5]{3}\\
	\mathrm{O} & \mathrm{O} & I_{d-(n+m-k)} \\
	\end{array}
	\right).
	$$
	But $\mathrm{rank~}(\gamma_{ig})_{\substack{k < i\leqslant m \\ k < g\leqslant m}}=m-k$ since
	\begin{equation*}
	\det J(\mathbf{\tilde y};\mathbf{y}) = \det \left(
	\begin{array}{c|c}
	I_{k} & \mathrm{O}_{k \times (m-k)} \\
	\hdotsfor[.5]{2}\\
	(\gamma_{ig})_{\substack{k < i\leqslant m \\1\leqslant g\leqslant k}} & (\gamma_{ig})_{\substack{ k < i\leqslant m \\k < g \leqslant m}}
	\end{array}
	\right)  = \det \left(
	\begin{array}{c|c}
	I_{k} & \mathrm{O}_{k\times (m-k)} \\
	\hdotsfor[.5]{2}\\
	\mathrm{O}_{(m-k)\times k} & (\gamma_{ig})_{\substack{ k < i\leqslant m \\k < g\leqslant m}}
	\end{array}
	\right)\neq 0.
	\end{equation*}
	Therefore ${\rm rank}\,J(\mathbf{\tilde x};\mathbf{x})=d$ and $\mathbf{\tilde x}$ is a permissible system of parameters at $p$ on $\widetilde D$.
\end{proof}

\section*{Acknowledgements}
I am deeply grateful to Professor Steven Dale Cutkosky for his generous guidance on gaining the knowledge of the subject and this result. Part of this work was completed in the helpful environment of Department of Mathematics, the University of Missouri which was supported by the Institute for Research in Fundamental Sciences of Iran. I am thankful to both institutions.

\end{document}